\documentclass[dvipsnames,11pt,reqno,twoside,final]{amsart}
\usepackage[a4paper,left=3cm,right=3cm,top=3cm,bottom=3cm]{geometry}
\usepackage[T1]{fontenc}
\usepackage[utf8]{inputenc}
\usepackage{xcolor}
\usepackage{amssymb,mathtools}  
\usepackage{dsfont}
\usepackage{enumitem}
\usepackage{subcaption}
\usepackage{booktabs}

\usepackage{preamble/mhequ}

\usepackage{amsthm}

\newcounter{counterEnvMain}
\newcounter{counterEnvDefault}
\numberwithin{counterEnvDefault}{section}

\theoremstyle{plain}

\newtheorem{lemma}[counterEnvDefault]{Lemma}
\newtheorem*{lemma*}{Lemma}

\newtheorem{theorem}[counterEnvDefault]{Theorem}
\newtheorem*{theorem*}{Theorem}

\newtheorem{proposition}[counterEnvDefault]{Proposition}
\newtheorem*{proposition*}{Proposition}

\newtheorem{corollary}[counterEnvDefault]{Corollary}
\newtheorem*{corollary*}{Corollary}

\newtheorem*{assumption*}{Assumption}

\theoremstyle{definition}

\newtheorem*{exercise*}{Exercise}

\newtheorem{definition}[counterEnvDefault]{Definition}
\newtheorem*{definition*}{Definition}

\newtheorem{notation}[counterEnvDefault]{Notation}
\newtheorem*{notation*}{Definition}

\newtheorem{remark}[counterEnvDefault]{Remark}
\newtheorem*{remark*}{Remark}

\newtheorem*{claim*}{Claim}

\newtheorem*{assertion*}{Assertion}

\usepackage{microtype}
\renewcommand\epsilon\varepsilon

\usepackage{hyperref}
\definecolor{colorlinks}{RGB}{0, 24, 168}
\definecolor{colorcites}{RGB}{0, 138, 118}
\hypersetup{
    colorlinks=true,
    linkcolor=colorlinks,
    citecolor=colorcites,
    urlcolor=colorlinks,
    pdfborder={0 0 0}
}

\usepackage{xargs}
\usepackage[colorinlistoftodos,prependcaption,textsize=tiny]{todonotes}

\newcommandx\work[2][1=]{\todo[linecolor=RoyalBlue,backgroundcolor=RoyalBlue!25,bordercolor=RoyalBlue,#1]{\textsc{todo} #2}}
\newcommandx\comment[2][1=]{\todo[linecolor=OliveGreen,backgroundcolor=OliveGreen!25,bordercolor=OliveGreen,#1]{\textsc{comment} #2}}
\newcommandx\mistake[2][1=]{\todo[linecolor=red,backgroundcolor=red!25,bordercolor=red,#1]{\textsc{mistake} #2}}
\newcommandx\improve[2][1=]{\todo[linecolor=orange,backgroundcolor=orange!25,bordercolor=orange,#1]{\textsc{improve} #2}}
\newcommandx\change[2][1=]{\todo[linecolor=yellow,backgroundcolor=yellow!25,bordercolor=yellow,#1]{\textsc{change} #2}}
\newcommandx\mem[2][1=]{\todo[linecolor=orange,backgroundcolor=orange!25,bordercolor=orange,#1]{\textsc{mem} #2}}
\newcommandx\status[2][1=]{\todo[linecolor=Blue,backgroundcolor=Blue!25,bordercolor=Blue,#1]{\textsc{Status} #2}}

\newcommand{\eps}{\epsilon}

\newcommand{\sfw}{\mathsf{w}}
\newcommand{\e}{\mathrm{e}}
\newcommand{\om}{\omega}

\newcommand\ind[1]{\mathds{1}_{\{#1\}}}

\newcommand{\bbLN}{N\mathbb{L}_{\scriptscriptstyle\boxtimes}^d}

\newcommand\C{\mathbb C}
\newcommand\D{\mathbb D}

\renewcommand\P{\mathbb P}

\newcommand\R{\mathbb R}

\newcommand\T{\mathbb T}

\newcommand\Z{\mathbb Z}

\newcommand\bbC{\mathbb C}

\newcommand\bbN{\mathbb N}

\newcommand\bbR{\mathbb R}

\newcommand\bbT{\mathbb T}

\newcommand\bbZ{\mathbb Z}

\newcommand\calC{\mathcal C}

\newcommand\calG{\mathcal G}

\newcommand{\sfP}{\mathsf{P}}

\mathtoolsset{showonlyrefs}

\makeatletter
\@namedef{subjclassname@2020}{\textup{2020} Mathematics Subject Classification}
\makeatother
\microtypesetup{expansion=false}

\title[Analyticity in FK-percolation]{Uniform analyticity of local observables in FK-percolation and analyticity of the Ising spontaneous magnetisation}

\author{Lucas D'Alimonte}
\address{LPSM, Sorbonne Université}
\email{dalimonte@lpsm.paris}

\author{Loïc Gassmann}
\address{Université de Fribourg}
\email{loic.gassmann@unifr.ch}

\date{\today}

\begin{document}

\maketitle

\begin{abstract}
We prove that, in the FK-percolation model, the probabilities of local events are uniformly analytic in the percolation parameter $p$ under suitable mixing assumptions on the measure, and satisfy a uniform exponential growth bound. 
This result allows us to prove that the magnetisation of the Potts model is analytic in a suitable range of parameters, including the Ising case in all dimensions $d \geq 3$ in the whole supercritical regime. 
We also provide a proof of the analyticity of the susceptibility of the Potts model with $q$ colours, for any $q \geq 2$ in the whole subcritical interval.
Finally, we prove the analyticity of various quantities in the FK-percolation measure, including the multi-point and truncated multi-point connectivity probabilities.  
\end{abstract}

\setcounter{tocdepth}{1}
\tableofcontents

\section{Introduction}
\noindent
The question of the regularity of thermodynamic quantities in lattice spin models is considered to be fundamental, and has attracted a lot of interest since the systematic study of these models. 
Indeed, in the physics literature, these regularity properties are used to characterize the nature of the phase transition exhibited by these models, if any.
A parameter at which the analyticity of a thermodynamic quantity is broken is considered by the physicists to be a phase transition point, and further properties of the phase transition can be derived by the way in which the analyticity is broken (to the best knowledge of the authors, these ideas were systematically introduced by Paul Ehrenfest~\cite{ehrenfest}). 
It is thus of deep interest to understand those regularity properties, both away and at the transition points. 

\medskip
\noindent
Another more specific reason for considering the question of the analyticity properties of thermodynamic quantities away from transition points is the discovery by Griffiths~\cite{Griffiths_singularities} of an intriguing phenomenon, now known as \emph{Griffiths singularities}.
In this seminal work, a phenomenon leading to a breaking of analyticity of the magnetisation and the free energy in models \emph{away from the critical point} was identified, in the setting of Ising models with random Hamiltonians, or \emph{diluted magnets} models. 
In \emph{pure} models, however, one generally expects the absence of Griffiths singularities, and a strict equivalence between analyticity breaking and point of phase transition.

 \medskip
 \noindent
This article is concerned with establishing analyticity properties of some thermodynamic observables of the \emph{Potts model with $q$ colours} in $\Z^d$. 
It is one the most studied lattice spin models, and includes as a special case when $q=2$ the celebrated \emph{Ising model}. 
As mentioned previously, the question of the off-critical analyticity has attracted a lot of attention during the past century. 
We shall however try to review the most significant contributions to this problem, some of them being of very different nature. 

\medskip
\noindent
Let us first discuss results concerned with analyticity of the \emph{free energy} (sometimes referred to as the \emph{pressure}, up to a $\beta$ multiplicative factor).
The first range of results concerns \emph{perturbative arguments}.  
Convergent cluster expansions have been used to establish analyticity properties in perturbative regimes (that is, very high or very low temperature Ising models) in~\cite{analyticity_pressure_ising_perturbative_1,analyticity_pressure_ising_perturbative_2}.
These type of arguments have been refined to general Gibbsian fields (including the Potts model) at very high temperatures in~\cite{analyticity_pressure_general}. 
General arguments regarding \emph{complete analyticity} of lattice spin models were provided in~\cite{complete_analyticity_Dobrushin} to argue that, under a set of strong assumptions including a strong mixing property, the free energy is analytic in the temperature. 
In two dimensions, perturbative arguments have been used in~\cite{complete_analyticity_2D_Potts} to establish \emph{complete analyticity} of the Potts model with large $q$ in the whole subcritical regime.
Moreover, the Ising model in two dimensions is exactly solvable, and its free energy has been computed by Onsager~\cite{Onsager}. In particular, it is indeed an analytic function except at the critical point. 
On the side of non-perturbative arguments, general Lee-Yang techniques~\cite{Lee_yang_1,Lee_yang_2} have been successful in proving the analyticity of the free energy in the magnetic parameter $h$ in the region $\Re(h) > 0$. 
The picture has finally been completed in the case of the Ising model by Ott~\cite{Ott_Ising,Ott_RCM} who showed that the free energy of the Ising model can be analytically continued around each non-critical inverse temperature $\beta$.  
Finally, we also mention the recent~\cite{lucas_piet_XY} which establishes the analyticity of the free energy of the XY model in its disordered phase, as well as the analyticity of the free energy of the 2D Coulomb gas in the complement of the BKT phase of the Villain model in 2D. 

\medskip
\noindent 
One may also consider different thermodynamic quantities, such as the \emph{spontaneous magnetisation} or the \emph{mass} of the model. 
In that case, the picture is much less complete and almost no non-perturbative argument is known, except for the two-dimensional Ising model, for which the mass was computed in~\cite{2D_Ising_mass} and the magnetisation in~\cite{2D_Ising_mag}, again defining analytic functions away from criticality. 
Actually, the article~\cite{Araki_analyticity_corr_Ising} shows that all the averages of local functions under the extremal Ising measures in two dimensions are analytic away from the critical point, using algebraic methods related to the $C^*$-algebras theory. 

However, it has been increasingly realised that \emph{graphical representations} of spin models~\cite{DC_graphical} are of great interest to pinpoint finer properties of these models, and this is the approach that we shall use in this article. 
In that direction, the first results were obtained in the influential~\cite{kesten_analyticity}, in which Kesten proved the analyticity of the susceptibility $\chi$ and of the number of open clusters per vertex $\kappa$ in the full subcritical regime of Bernoulli percolation. 
Later on, analyticity of the \emph{mass} (or \emph{inverse correlation length}) of subcritical Bernoulli percolation was established in~\cite{Chayes_analyticity_corr_Bernoulli}.
We also mention the proof of the analyticity of the mass of the self-avoiding walk in the subcritical regime of~\cite{Ioffe_analyticity_SAW}.

In the Edwards--Sokal coupling of the Potts model, the magnetisation is related to the probability in the FK-percolation that the origin is contained in an infinite cluster. 
This probability is often called $\theta(p)$, when $p$ is the percolation parameter, and is non-zero in the \emph{supercritical} regime of percolation. 
The question of the analyticity of $\theta$ is thus of deep interest, and we note that it was actually raised by Kesten in~\cite{kesten_analyticity}. 
This question was answered in the case of Bernoulli percolation on $\Z^d$ in the work~\cite{PG_analyticity}, in which the analyticity of the $\theta$ function of Bernoulli percolation in the full supercritical regime is proved.
For Bernoulli percolation beyond the graph $\Z^d$, we may cite the very recent~\cite{martineau_panagiotis_locality}, that extends the precedent result to Bernoulli percolation on a much more general class of graphs of polynomial growth.
Finally we may mention the closely related work~\cite{panagiotis_severo_analyticity_GFF}, which proves analyticity of the $\theta$ function of yet a different model called the \emph{Gaussian free field percolation}.

\medskip
\noindent
In this article, we investigate these questions for a classical graphical representation of the Potts model called \emph{FK-percolation}, that generalizes Bernoulli percolation and the edges of which are no longer independent. 
The analyticity of some quantities of FK-percolation naturally implies similar results in the case of the Potts model by the so-called \emph{Edwards--Sokal coupling}.
In particular, we rigorously prove that:
\begin{enumerate}
\item The susceptibility of the Potts model is analytic in the whole subcritical regime, in any dimension, for any $q \geq 2$. 
\item The spontaneous magnetisation of the Ising model is analytic in dimension $d \geq 3$ in the whole supercritical regime, answering Kesten's question~\cite{kesten_analyticity} for the FK-percolation with $q=2$ (also known as the \emph{FK-Ising} model),
\item The multi-point connectivity probabilities of the FK-Ising measure are analytic in dimension $d \geq 3$ when $p \neq p_c$. 
\item Local observables of FK-percolation are uniformly analytic and have \emph{exponentially bounded growth} in a suitable range of parameters, that includes the subcritical regime for any $q\geq 1$, and the whole off-critical regime for $q=2$. 
\end{enumerate}
The method consists in refining the cluster expansion method developed in~\cite{Ott_Ising,Ott_RCM} to prove item (4). 
The proof of items (1), (2) and (3) is then done using a straightforward adaptation of~\cite{PG_analyticity}.

\subsection{Definition of the models}

In this work, we shall work on the graph $\Z^d$, with $d \geq 2$, endowed with its nearest-neighbour graph structure. 
The nearest-neighbour connectivity will be denoted by the symbol $\sim$.
For $n \in \Z_{\geq 1}$, we define the box $\Lambda_n \coloneqq [-n, n]^d \cap \Z^d$.

\subsubsection{The Potts model}

The Potts model is a celebrated example of lattice spin model (we refer to~\cite{HDC_lecture_notes} for an extensive review). 
In $\Z^d$, with $d\geq 1$, a formal definition is given below. First set an integer $q \geq 2$ (the number of colours).
The spin space will be the set $\{1, \dots, q\}$. Given an inverse temperature $\beta > 0$,  a finite set $\Lambda \subset \Z^d$, and a \emph{boundary condition} $\tau \in \{0,1, \dots, q\}^{\bbZ^d}$, define the probability measure:
\[
\mu^\tau_{\Lambda, \beta}[\sigma] = \frac{1}{\mathbf{Z}^{\tau}_{\Lambda, \beta}} \exp\big( \beta \sum_{x\sim y \atop x,y \in\Lambda} \ind{\sigma_x = \sigma_y} + \beta \sum_{x\sim y \atop x \in \Lambda, y \notin\Lambda} \ind{\sigma_x = \tau_y}\big).
\]
The constant $\mathbf{Z}^{\tau}_{\Lambda, \beta}$ is the unique number normalizing the measure to 1, and is called the \emph{partition function}. 
We will mostly be interested in \emph{monochromatic} measures. 
For $j\in \{0, 1, \dots, q\}$, abbreviate $\mathbf{j}$ for the boundary condition constant equal to $j$ on $\Z^d$. 
The measure $\mu^{\mathbf{0}}_{\Lambda, \beta}$ will be referred to as \emph{free} in the volume $\Lambda$, while $\mu^{\mathbf{j}}_{\Lambda, \beta}$ with $j \neq 0$ will be referred to as \emph{monochromatic} in volume $\Lambda$. 
It is classical (see ~\cite{HDC_lecture_notes}) that all the measures $\mu^{\mathbf{j}}_{\Lambda, \beta}$ admit limits when $\Lambda \rightarrow \Z^d$, provided a suitable growth condition on the volumes $\Lambda$ (the sequence of boxes $(\Lambda_n)_{n \geq 1}$ satisfies this growth condition). 
Those limits will be referred to as $\mu^{\mathbf{j}}_{\beta}$.

When $d \geq 2$, the Potts model in $\Z^d$ notoriously exhibits a \emph{phase transition} phenomenon. 
Define the \emph{spontaneous magnetisation}\footnote{The choice of $\mathbf{1}$ for the boundary condition matching the event $\{\sigma_0 = 1\}$ is conventional: by spin symmetry, this quantity is the same if one replaces the boundary condition by $\mathbf{j}$ and the considered event by $\{\sigma_0 = j\}$.} as
\[
m^*(\beta) \coloneqq \mu^{\mathbf{1}}_\beta[ \sigma_0 = 1 ] - \frac 1 q. 
\]
The \emph{critical parameter} is defined as $\beta_c \coloneqq \sup\{ \beta > 0, m^*(\beta) = 0 \}$.
Then, when $d \geq 2$, a celebrated and classical argument due to Peierls implies that for any $q \geq 2$, one has $\beta_c \in (0, +\infty)$. 

We will also be interested in the so called \emph{susceptibility function}, defined for any $\beta < \beta_c$ as follows:
\[
\chi_{\mathsf{Potts}}(\beta) \coloneqq \sum_{x \in \Z^d} \{ \mu^{\mathbf{1}}_\beta[ \sigma_0 = \sigma_x] - \tfrac 1 q \}.
\] 
When $\beta < \beta_c$, the susceptibility is finite in any dimension $d \geq 2$; this is a consequence of the deep results of~\cite{DCRT_sharpness}.

\subsubsection{The random-cluster model}

The Potts model is classically related to a percolation model called \emph{FK-percolation} (sometimes referred to as \emph{random-cluster model}), that we now define. 
Slightly abuse notation by writing $\Z^d \coloneqq (V(\Z^d), E(\Z^d))$, where $V(\Z^d)$ denotes the vertex set of $\Z^d$ and $E(\Z^d)$ the set of nearest-neighbour unoriented edges of $\Z^d$. 
A \emph{percolation configuration} is a configuration of $\{0,1\}^{E(\Z^d)}$. 
As previously, we start by defining the measure in finite volumes. Let $G= (V(G), E(G))$ be a finite subgraph of $\Z^d$.
For any two percolation configurations $\om, \xi \in \{0,1\}^{E(\Z^d)}$, define $k^\xi_G(\om)$ to be the number of vertex connected components (hereafter referred to as \emph{clusters}) of $\om_{|E(G)} \cup \xi_{|E(\Z^d)\setminus E(G)}$ intersecting $V(G)$. 
For any two parameters $p\in (0,1), q \geq 1$, define the random cluster measure in $G$ with boundary condition $\xi$ as follows: 
\begin{equation}\label{equ:def_RCM}
\phi^{\xi}_{G,p ,q}[\om] = \frac{1}{\mathbf{Z}^\xi_{G,p,q}} \left(\frac{p}{1-p} \right)^{|\om|} q^{k^\xi_G(\om)},
\end{equation}
where $\mathbf{Z}^\xi_{G,p,q}$ is a renormalisation function. Note that this function is a polynomial in $\frac{p}{1-p}$. As previously, when the boundary conditions are either \emph{free} (\emph{i.e.} $\xi \equiv 0$ on $\Z^d$) or \emph{wired} (\emph{i.e.} $\xi \equiv 1$ on $\Z^d$), the measures can be extended to the full lattice by taking appropriate limits $G \nearrow \Z^d$. 
Those measures will respectively be referred to as $\phi^0_{p,q}$ and $\phi^1_{p,q}$. 

Another type of boundary conditions will also be considered in this article. 
Define $\T_N$ to be the torus of radius $N$.
For a percolation configuration $\om \in \{0,1\}^{E(\T_N)}$, define $k_{\T_N}(\om)$ to be the number of open clusters of $\om$, where the connectivity notion is inherited from the torus connectivity.
Then, we define the FK-percolation measure with \emph{periodic} boundary conditions $\phi_{\T_N, p, q}$ by~\eqref{equ:def_RCM}, with $k_{\T_N}$ in place of $k^\xi_G$.

This model also undergoes a phase transition for the existence of an infinite cluster: indeed, first define 
\[
\theta(p) \coloneqq \phi_{p,q}^1[ 0 \leftrightarrow \infty],
\] 
where $\{0 \leftrightarrow \infty\}$ denotes the event that the cluster of 0 has infinite cardinality. 
Further define $p_c \coloneqq \sup\{p \in [0,1], \theta(p) = 0\}$. Then, when $d \geq 2$, the above-mentioned \emph{Peierls' argument} allows to prove that $p_c \in (0,1)$.

In this setting, the susceptibility is defined for $p<p_c$ as 
\[
\chi_{\mathsf{FK}}(p) \coloneqq \sum_{x\in\Z^d} \phi_{p,q}^1[0 \leftrightarrow x] = \phi^1_{p,q}[|\calC_0|],
\]
where $\calC_0$ denotes the connected component of 0.
Again, when $p < p_c$, the susceptibility is finite in any dimension $d \geq 2$ due to~\cite{DCRT_sharpness}.

Moreover, the \emph{multi-point connectivity functions} are defined as follows, for $p \in (0,1)$, $q \geq 1, k \in \Z_{\geq 2}$, and $x_1, \dots, x_k \in \Z^d$:
\begin{itemize}
\item $\tau_{p,q}(x_1, \dots, x_k) = \phi^1_{p,q}[x_1, \dots, x_k \text{ lie in the same cluster}]$,
\item $\tau_{p,q}^\mathsf{f}(x_1, \dots, x_k) = \phi^1_{p,q}[x_1, \dots, x_k \text{ lie in the same \emph{finite} cluster}]$.
\end{itemize} 

\subsubsection{Edwards--Sokal coupling}

It is well-known that the Potts model with $q$ colours at inverse temperature $\beta > 0$ and the random-cluster model with parameters $p(\beta) \coloneqq 1 - \e^{-\beta}$ and $q$ can be coupled through the so-called \emph{Edwards--Sokal coupling} (again we refer to~\cite{HDC_lecture_notes} for a more extensive background). 
In particular, it follows from the properties of the coupling that for any $G \subset \Z^d$ finite, 
\[
\mu^{\mathbf{1}}_{G, \beta, q}[ \sigma_0 = 1] = \frac 1 q + \frac{q-1}{q}\phi^1_{G,p(\beta),q}[0 \leftrightarrow \partial G ].
\]
Letting $G \nearrow \Z^d$, we obtain the following identity that relates the magnetisation of the Potts model and the function $\theta$ in the random-cluster model:
\begin{equation}\label{equ:edwards_sokal}
m^*(\beta) = \frac{q-1}{q}\theta(p(\beta)).
\end{equation}

\subsection{Exposition of the results}

As our result concerns different regimes simultaneously, it will be convenient to introduce the set of ``good'' parameters to which the results apply.
The reader might note that the following definition is extracted from~\cite{Ott_RCM}. 

\begin{definition}\label{def:good_parameters}
Define $\calG_{\mathsf{FK}}$ to be the set of triplets $(d,p,q)$ such that the FK measures in finite volumes $\phi_{\Lambda_n, p, q}$ satisfy the following hypotheses. 
\begin{enumerate}
\item \emph{Exponential weak mixing property}: for any $\alpha > 1$, there exists $c > 0$ such that for any $n \geq 0$ sufficiently large\footnote{$\mathrm{d}_{\mathsf{TV}}$ denotes the classical total variation distance between probability measures, and for any measure $\mu$, $\mu( \cdot_{|\Lambda_n})$ denotes the restriction of the measure $\mu$ to the box $\Lambda_n$.}, 
\[
\mathrm{d}_{\mathsf{TV}}( \phi^1_{\Lambda_{\alpha n}, p,q}( \cdot_{|\Lambda_n} ), \phi^0_{\Lambda_{\alpha n}, p,q}( \cdot_{|\Lambda_n} ) ) < \exp(-cn). 
\]
\item One of the two following conditions holds:
\begin{enumerate}
\item \textit{Exponential decay}: for any $\alpha > 1$, there exists $c>0$ such that for any $n \geq 0$ large enough, 
\[
\phi^1_{\Lambda_{\alpha n, p ,q}}[A_n] < \exp(-cn),
\]
where $A_n$ is the event that $\Lambda_n$ contains a cluster of diameter greater than $n/100$.
\item \textit{Local uniqueness}: for any $\alpha > 1$, there exists $c>0$ such that for any $n \geq 0$ large enough, 
\[
\sup_\eta \phi^\eta_{\Lambda_{\alpha n, p ,q}}[U_n^c] < \exp(-cn),
\]
where the local uniqueness event $U_n$ is the event that $\Lambda_n $ contains a \emph{crossing} cluster and all the other clusters have diameter less than $n/100$, where a crossing cluster is a cluster that connects all the sides of the box $\Lambda_n$. The event $U_n^c$ is the complement of $U_n$.
In the last property, the supremum is taken over all the boundary conditions on the box $\Lambda_{\alpha n}$. 
\end{enumerate}
\end{enumerate}
\end{definition}

The exact form of the events considered for hypothesis $2.$ is not so important, they are chosen in order to implement a classical renormalization scheme.

We are now ready to state our main results.
The key result concerns analyticity of local events in the FK measure. 
In the next statement, the notation $\phi_{\bbT_N, p+z, q}$ denotes the analytic extension of the finite-volume measure $\phi_{\bbT_N, p, q}$.
\emph{A priori}, this is an ill-defined object, as the complex partition function used to renormalise this measure could vanish and induce a singularity on the finite-volume probabilities\footnote{This is a key difference with the case of Bernoulli percolation, treated in \cite{PG_analyticity}. Indeed, in Bernoulli percolation, the probability of a local event is a polynomial in $p$, which directly defines an \emph{entire} function in $\C$. In the FK measure, the situation is more complex due to the presence of the partition function, and even defining those complex extensions in finite-volume requires some work.}. 
The statement has to be understood as ruling out this possibility in the considered $\delta$-neighbourhood of $p$. 

In the statement, we say that a function $F : \{0,1\}^{E(\bbZ^d)} \rightarrow \bbR$ is \emph{local} if, for any $\om \in \{0,1\}^{E(\bbZ^d)}$, the value of $F(\om)$ is determined by the restriction of $\om$ to a finite subset of $E(\bbZ^d)$ that we refer to as the \emph{support} of $F$. Observe that by definition, a local function is always bounded. 
For $\delta > 0$, let $\D_\delta(0) \coloneqq \{z \in \C, |z| < \delta \}$, and denote by $\vert \cdot \vert$ the cardinality of a set. 
\begin{theorem}\label{thm:local_analyticity}
Assume that $(d,p,q) \in \calG_{\mathsf{FK}}$. Then, there exists $\delta > 0$ such that:

\begin{enumerate}
\item For any \emph{local} function $F$, for any $z \in \C$ satisfying $|z|<\delta$, the following limit exists:
\[
\phi_{p + z, q}[F(\om)] \coloneqq \lim_{N\rightarrow \infty} \phi_{\T_N, p + z, q}[F(\om)] \in \C.
\]
Furthermore, the function 
\[ z \in \D_\delta(0) \mapsto \phi_{p + z, q}[F(\om)]\in \C \]
 is complex analytic.
\item 

There exists a sequence $c_\eps$, going to 0 when $\eps$ goes to 0, such that for any $|z| < \eps < \delta$ and any non-negative \emph{local} function $F$ which is not identically 0, 
\begin{equation}\label{equ:exp_bound_theorem}\tag{$\star$}
\left| \frac{\phi_{p + z, q}[F(\om)]}{\phi_{p, q}[F(\om)]} \right| \leq 2\exp(c_\eps |\mathsf{Supp}_F|),
\end{equation}
where $\mathsf{Supp}_F$ is the support of $F$.

 \end{enumerate}
\end{theorem}

\begin{remark}[Example of ``good'' parameters]\label{rem:good_parameters}
The set $\calG_{\mathsf{FK}}$ includes the following triplets. 
\begin{itemize}
\item $d\geq 2$, $q \geq 1$ and $p < p_c$,
\item $d = 2$, $q \geq 1$ and $p > p_c$,
\item $d \geq 3$, $q=2$ and $p > p_c$,
\end{itemize}
The first point follows from the subcritical sharpness of FK-percolation, proved in~\cite{DCRT_sharpness}. The second point follows from the first point and duality. The third point is more subtle: the local uniqueness (item (2.b)) follows by combining the celebrated argument from Bodineau~\cite{bodineau} with the renormalisation tools developed by Pisztora \cite{Pisztora_CG} (see also~\cite{severo_slab_perco} for a more direct argument) whereas the mixing property (item (1)) is the main result of~\cite{DC_sup_sharpness_ising}. It is worth pointing out that one of the main ingredients of~\cite{DC_sup_sharpness_ising} is a renormalisation scheme that heavily relies on \cite{bodineau}. 

We note that it is expected that $\calG_{\mathsf{FK}}$ coincides with the set of \emph{off-critical} parameters for the FK measure, with property (2.a) (resp. (2.b)) of Definition~\ref{def:good_parameters} corresponding to subcritical (resp. supercritical) parameters. 
This deep conjecture still remains to be proved when $d \geq 3$ and $p > p_c$. Let us mention that proving item (2.b) corresponds to proving \emph{supercritical sharpness} and also that item (1) is not expected to follow directly from item (2.b) (see the open problem discussion in~\cite{DC_sup_sharpness_ising}).
Let us conclude this remark by mentioning that supercritical sharpness has been proved to hold in a handful of related models, such as Voronoi percolation~\cite{dembin_severo_sup_sharpness_voronoi} and level-set percolation of the discrete Gaussian free field~\cite{duminil_sup_sharpness_GFF}. 
Maybe more related to this work is the case of the discrete $\varphi^4$ model, treated in the recent~\cite{panis_supercritical_sharpness_phi_4}.
\end{remark}

On the side of the Potts model, we define $\calG_{\mathsf{Potts}}$ to be the set of parameters $(d,\beta, q)$ such that $(d, p(\beta), q) \in \calG_{\mathsf{FK}}$, where we recall that $p(\beta) = 1 - \e^{-\beta}$.
The previous theorem allows us to derive the following result, concerning the analyticity of the magnetisation of the Potts model and of the function $\theta$ in the FK measure.
We note that in the case $q=2$, it totally determines the domain of analyticity of $\theta$. As mentioned earlier, the question was originally raised by Kesten~\cite{kesten_analyticity} in the context of Bernoulli percolation. 

\begin{theorem}\label{thm:analyticity_magnetisation}
The following are true.
\begin{itemize}
\item Let $(d,p_0,q) \in \calG_{\mathsf{FK}}$. 
Then, the function $p \in (0,1) \mapsto \theta(p)$ admits an analytic extension around $p_0$. 
\item Let $(d,\beta_0 ,q) \in \calG_{\mathsf{Potts}}$. 
Then, the function $\beta > 0 \mapsto m^*(\beta)$ admits an analytic extension around $\beta_0$. 
\end{itemize}
\end{theorem}

We highlight one important consequence of the result concerning the Ising model.

\begin{corollary}
For any $d \geq 2$, the spontaneous magnetisation of the Ising model $ m^*$ is real analytic on $\R_{> 0}\setminus\{\beta_c\}$.
\end{corollary}

We emphasize that the result is new only for dimensions $d \geq 3$, since in two dimensions it is known since the work~\cite{2D_Ising_mag}.
Our next result regards the analyticity of the susceptibility function of the random-cluster model. 

\begin{theorem}\label{thm:analyticity_susceptibility}
The following are true.
\begin{itemize}
\item For any $d\geq 1$, any $q \in \Z_{>1}$, any $\beta_0 < \beta_c$, the function $\beta \geq 0 \mapsto \chi_{\mathsf{Potts}}(\beta)$ admits an analytic extension around $\beta_0$. 
\item For any $d\geq 1$, any $q \geq 1$, any $p_0 < p_c$, the function $p \in (0,1) \mapsto \chi_{\mathsf{FK}}(p)$ admits an analytic extension around $p_0$. 
\end{itemize} 
\end{theorem}

Our last result concerns the analyticity of the multi-point connectivity probabilities in the FK measure.  

\begin{theorem}\label{thm:multi_point_correlations}
Let $(d,p_0,q) \in \calG_{\mathsf{FK}}$. 
For any $k \in \Z_{\geq 2}$, any points $x_1, \dots, x_k \in \Z^d$, the functions
\[ p \in (0,1) \mapsto \tau_{p,q}(x_1, \dots, x_k) \qquad \text{ and }\qquad p \in (0,1) \mapsto \tau^\mathsf{f}_{p,q}(x_1, \dots, x_k)  \] 
admit an analytic extension around $p_0$. 
\end{theorem}

The identity $\mu^{\mathbf{1}}_{\beta, q}[\sigma_x = \sigma_y] = \phi^1_{p(\beta), q}[x \leftrightarrow y]$ given by the Edwards--Sokal coupling does not generalize easily for multi-point correlations of the Potts model. 
Nonetheless, in the case of the truncated two-points correlation function of the Ising model, we are able to get the following result\footnote{The slightly strange form of this observable is due to our definition of the Potts model. This observable corresponds to the classical two-points truncated Ising correlation $\langle \sigma_x;\sigma_y \rangle_\beta^+ = \langle \sigma_x\sigma_y \rangle_\beta^+ - \langle\sigma_x\rangle_\beta^+ \langle\sigma_y\rangle_\beta^+$.}: 
\begin{corollary}\label{cor:trunc_two_pointsl}
Let $d \geq 2$ and $\beta \neq \beta_c$. Then, for any $x,y \in \Z^d$, the Ising observable
\[ \beta > 0 \mapsto \langle \sigma_x;\sigma_y \rangle^+_\beta \coloneqq \mu^\mathbf{1}_{\beta, q=2} [ (2\sigma_x - 1)(2\sigma_y -1) ] -  \mu^\mathbf{1}_{\beta, q=2} [ (2\sigma_x - 1)]\mu^\mathbf{1}_{\beta, q=2} [ (2\sigma_y- 1)] \]
is analytic around $\beta$.
\end{corollary}

Again, the result is new only for dimensions $d \geq 3$, as the case $d=2$ was solved in~\cite{Araki_analyticity_corr_Ising}.
It is remarkable that the method used in the above-mentioned article is somehow totally orthogonal to the one used in this work. 

\begin{proof}
Let $d \geq 2$, $q=2$, and $\beta \neq \beta_c(q)$, so that $(d,\beta, q)\in \mathcal G_{\mathsf{Potts}}$. 
The Edwards--Sokal coupling implies that:
\[ \langle \sigma_x;\sigma_y \rangle^+_\beta = \phi_{p(\beta), 2}[ x \leftrightarrow y ] -  \theta(p(\beta))^2. \]
Both of those terms are analytically extendable around $\beta$ by Theorems~\ref{thm:analyticity_magnetisation} and~\ref{thm:multi_point_correlations}
\end{proof}

\begin{remark}
As previously mentioned, the generalization of this result to truncated multi-point correlations, that would allow \emph{e.g.} the proof of the analyticity of the so-called \emph{Ursell functions} of the Ising model is complicated by the fact that the events appearing in the Edwards--Sokal become more convoluted than ``straightforward'' connectivity events such as described by the functions $\tau$ and $\tau^\mathsf{f}$. 
Nonetheless, one might expect that a generalization of the method of~\cite{PG_analyticity} that would allow to prove the analyticity of the probability of those events might still holds for the FK measure.
Another possible way to treat that problem would be to analyse the \emph{double random current measure} of the Ising model, which is much more tailored to encode the truncated correlation functions of the Ising model. 
In that case, very recent results on the mixing properties of this measure~\cite{ulrik_supercritical_RCurrM} might allow to successfully replicate our methods. 
In any case, we refrained to investigate in that direction to keep the paper relatively short. 
\end{remark}

\subsection{Discussion of the results and of the proof}\label{sec:remarks_statement}
\subsubsection{Comments on the statements}
We start with some comments concerning our main results. 

\begin{itemize}
\item The main innovation of this result is the control of the exponential growth of the probability of a local event under the complex measure given by~\eqref{equ:exp_bound_theorem}. 
Indeed let us mention that point (1) of Theorem~\ref{thm:local_analyticity} could be proved with a fairly simpler argument that the one we use in this paper. 
However, we will derive both items simultaneously from Proposition~\ref{prop:unif_bound_ratio_probas}. 
The bound~\eqref{equ:exp_bound_theorem} can be interpreted as a \emph{complex finite energy} property: when changing the percolation parameter by a small complex amount, the probability of a local event is only affected by a small complex amount per vertex present in its support.
In the Bernoulli case ($q=1$), this bound is actually trivial, as the probability of a local event is polynomial, and is the key ingredient in classical proofs of the analyticity of thermodynamic observables for this model (see~\cite{PG_analyticity}). 
Our result extends this locality property to the dependent case ($q > 1$) and may be seen as a generic tool to extend proofs of analyticity of thermodynamic observables from Bernoulli percolation to FK-percolation.

\item The reader might be worried about the specific choice of toroidal boundary conditions. 
However, this choice is arbitrary and convenient for the proofs.
As will be noted in the proof of Theorem~\ref{thm:local_analyticity}, there exists a unique infinite-volume measure around $p$ for this choice of parameters, and $\phi_{p+z, q}$ is the analytic extension of this infinite-volume measure. 
As such, one may \emph{a posteriori} replace the boundary conditions by free or wired in the box $\Lambda_N$, without affecting the value of $\phi_{p+z, q}$. 

\item While the Ising model is integrable in dimension 2 and the magnetisation explicitly computable in that setting~\cite{2D_Ising_mag}, this integrability property fails when $d \geq 3$.
In this sense, the analyticity property proved in this article is the best ``rigidity'' result one could expect in absence of an explicit computation.

\end{itemize}

\subsubsection{Overview of the proof and organisation}

As mentioned in the introduction, the proofs of Theorems~\ref{thm:analyticity_magnetisation},~\ref{thm:analyticity_susceptibility}, and~\ref{thm:multi_point_correlations} are fairly standard, once Theorem~\ref{thm:local_analyticity} is established. 
Indeed, the analyticity of the FK susceptibility follows from the fact that the cluster volume enjoys the \emph{Aizenman--Newman--Barsky} property (\emph{i.e.} exponential decay of the volume of the cluster of 0 in the subcritical phase).
The proof of the analyticity of $\theta$ in the supercritical regime is much more involved, as the Aizenman--Newman--Barsky property for finite clusters is no longer true in the supercritical regime.
However it was observed in~\cite{PG_analyticity} that a much more refined notion of~\emph{separating component} do enjoy this property, and that thanks to a combinatorially involved inclusion-exclusion principle, one may prove the analyticity of $\theta$, provided a good exponential control on the growth rate of the analytic extensions of the probabilities of local events. 
This strategy was then successfully implemented by these authors in the case of Bernoulli percolation (corresponding to $q=1$ in the random-cluster model). 
The conclusion of this discussion is that to extend these methods to FK-percolation, the crux is Theorem~\ref{thm:local_analyticity} and especially the bound given by~\eqref{equ:exp_bound_theorem}.
The results concerning the Potts model trivially follow from the Edwards--Sokal coupling. 
The proof of Theorem~\ref{thm:local_analyticity} is done in Section~\ref{sec:unif_analyticity}, while the standard proofs of Theorems~\ref{thm:analyticity_magnetisation},~\ref{thm:analyticity_susceptibility}, and~\ref{thm:multi_point_correlations} are sketched in Section~\ref{sec:mag_susc}.

For proving Theorem~\ref{thm:local_analyticity}, we rely on two main ingredients. 
The first one is the observation by Ott~\cite{Ott_RCM}, based on ideas coming from~\cite{sly_dependency} and~\cite{harel_spinka} that a construction based on the coupling from the past of the Glauber dynamics for FK-percolation allows to construct a~\emph{dependency percolation with exponentially small clusters}.
Following the strategy of~\cite{Ott_RCM}, we use this dependency percolation as a blackbox to apply cluster expansion methods. 

However, it turns out that one faces severe difficulties in implementing ``direct'' cluster expansion methods for local observables of the FK-percolation, and that any such attempt will produce a radius of analyticity that depends on the size of the support of the observable that one tries to expand (and that this radius will shrink to 0 when the size of the support increases). In particular, expanding observables as one single polymer partition function will lead to such a phenomenon, as illustrated in the proof of Theorem 2.4 in~\cite{Ott_RCM}, for instance.

Our main contribution is to overcome this difficulty. The method starts by expanding complex perturbations of local observables as \emph{weighted sums} of polymer partition functions instead of a single one. 
The main effort is then dedicated to control the complex modulus of the weight function to produce a global upper bound on the modulus of those complex perturbations.
The paper is organised as follows.

\begin{itemize}
\item Section~\ref{sec:unif_analyticity} is devoted to the proof of Theorem~\ref{thm:local_analyticity}.
\begin{itemize}
\item In Section~\ref{subsec:resumm_weighted}, the expansion as a weighted sum of polymer partition functions is performed. 
\item In Section~\ref{subsec:clust_exp_2}, we implement classical cluster expansion bounds on each one of those polymer partition functions.
\item In Section~\ref{subsec:exp_summability}, we analyse the weighting function and prove its ``exponentially resummability''. 
\item In Section~\ref{sec:proof_thm_local}, we use those tools to provide a proof of Theorem~\ref{thm:local_analyticity}.
\end{itemize}
\item Section~\ref{sec:mag_susc} makes use of Theorem~\ref{thm:local_analyticity} to prove both Theorems~\ref{thm:analyticity_magnetisation},~\ref{thm:analyticity_susceptibility}, and~\ref{thm:multi_point_correlations}.
\item Technical tools about the cluster expansion are deferred to the Appendix.
\end{itemize}

\section{Uniform analyticity of local observables of the FK measure}\label{sec:unif_analyticity}

The proof of Theorem~\ref{thm:local_analyticity} crucially relies on the concept of \emph{dependency encoding measure}, itself closely related to the notion of \emph{finite factor of i.i.d}~\cite{harel_spinka}. 
The existence of such a dependency encoding measure was proved in~\cite{Ott_Ising, Ott_RCM} respectively in the case of the Ising model and in the case of FK-percolation in a suitable range of parameters. 
In these articles, this dependency encoding measure is constructed based on the mixing properties of the model, combined with coupling from the past for its Glauber dynamics. 
In the next subsection, we directly import the main results of~\cite{Ott_RCM}, that will be relevant for our analysis. 

\subsection{Dependency encoding measure for FK-percolation: a recap of~\cite{Ott_RCM}}

Fix $N, L \in \Z_{> 0}$ such that $N$ is an integer multiple of $L$. 
As mentioned earlier, in the regime of parameters under consideration, all the infinite-volume limit measures are equal, and it will be convenient to work in the $d$-dimensional torus of radius $N$, denoted by $\T_N$. 
Since $q$ is a fixed parameter, we will often omit it in the notation.
Finally, we also introduce the $L$-\emph{coarse-grained torus} of size $N$, defined as $\T_N^L \coloneqq  ((2L)\Z)^d \cap \T_N$. 
This set inherits the graph structure of $\bbT_N$ by placing edges between vertices at distances $2L$ in $\T_N^L$.   
Finally define $\sfP^L_N \coloneqq  \{ \gamma \subset \T_N^L, \gamma \text{ connected and non-empty}\}$. 
The main result of~\cite{Ott_RCM} is then the following (in the statement, all the constants are uniform in $N$, once $L$ is chosen large enough and $L$ divides $N$). 

\begin{theorem}[Theorem 4.1,~\cite{Ott_RCM}]\label{thm:dependency_encoding}
Let $(d,p,q) \in \calG_{\mathsf{FK}}$. 
Then there exists $L_0 > 0$ such that for any $L \geq L_0$, there exists a probability measure $\P$ on $  \{0,1\}^{E(\T_N)} \times (\sfP^L_N)^{\T_N^L}$ satisfying the following properties.
When $(\om, (\calC_x)_{x \in \T_N^L})  \sim \P$, then:
\begin{enumerate}
\item $\om$ is distributed according to $\phi_{\bbT_N,p}$.
\item For any $x\in\T_N^L$, $\calC_x$ contains $x$. 
\item For any $\Delta_1, \Delta_2 \subset \T_N^L$, any $F,G$ functions of $\om$ supported on $\Delta_1$ and $\Delta_2$ respectively and any $C_1 \supset \Delta_1, C_2 \supset \Delta_2$ such that $C_1 \cap C_2 = \emptyset$, then
\[
\P[F(\om)G(\om) \ind{\calC_{\Delta_1} = C_1} \ind{\calC_{\Delta_2 = C_2}} ] = \P[F(\om)\ind{\calC_{\Delta_1} = C_1} ] \P[G(\om)\ind{\calC_{\Delta_2 = C_2}} ],
\]
where we call $\calC_\Delta \coloneqq  \bigcup_{x\in\Delta}\calC_x$. 
\item There exist two sequences $\tilde c_L, c_L > 0$ going to infinity when $L$ goes to infinity such that, for any $x \in \T_N^L$ and any $n \in \Z_{> 0}$, 
\[  \P[|\calC_x| \geq n] \leq \tilde c_L\exp(-c_L n). \]
\end{enumerate}
\end{theorem}

We also import an easy combinatorial consequence of the fourth item of the latter result which will be useful for the proof of Theorem~\ref{thm:local_analyticity}.

\begin{lemma}[Corollary 4.2,~\cite{Ott_RCM}]\label{lem:exp_dec_cluster_sets}
Let the conditions of Theorem~\ref{thm:dependency_encoding} be satisfied. 
There exists a sequence $a_L > 0$ such that for any $\Delta \subset \T_N^L$ and any $n\in \Z_{> 0}$,
\[
\P[|\calC_\Delta| \geq n] \leq \exp(-c_L n + a_L|\Delta|)
\]
where $c_L$ was defined in Theorem \ref{thm:dependency_encoding}.
\end{lemma} 

Those two results imported, the proof can be started.

\medskip
{\centering
\textbf{In what follows, we fix $(d,p,q)\in \calG_{\mathsf{FK}}$.  $L$ will always be chosen larger than $L_0$ so that Theorem~\ref{thm:dependency_encoding} applies, and $\P$ will denote the probability measure constructed in this result. We also fix $N \geq L$.}
}
\subsection{Resummation as a weighted sum of polymer partition functions}\label{subsec:resumm_weighted}

We fix a function $F$ satisfying the requirement of point (2) of Theorem \ref{thm:local_analyticity}. 
This means that $F$ is a non-negative local function not identically zero. 
We also denote the support of $F$ by $\Delta$. 

Observe that $\phi_{\T_N, p}[F(\om)]$ is a rational function of the parameter $p$. Hence, it defines a holomorphic function from $\bbC$ to the Riemann sphere (the poles correspond to the zeroes of the partition function and get mapped to $\infty$).  With a slight abuse of notation, we will write $\phi_{\T_N, p + z}[F(\om)]$ for the image of the extension at a given point $p+z \in \bbC$.

We introduce the following quantity, for $z \in \C$ satisfying\footnote{This assumption will be implicit throughout all the paper. If needed, the convergence radius needs to be decreased accordingly.} $|z| < 1-p$:
\[
\alpha_z \coloneqq  (1+ \tfrac z p)(1 - \tfrac{z}{1-p})^{-1}.
\]
The following lemma shows that the analytic extension $\phi_{\T_N, p + z}$ corresponds to the measure $\phi_{\T_N,p}$ biased by $\alpha_z^{|\omega|}$.

\begin{lemma}\label{lem:rewriting_phi_loc}
For any $z \in \C$ such that $\phi_{\T_N,p}[\alpha_z^{|\om|} ] \neq 0$,
\[
\phi_{\T_N, p + z}[ F(\om) ] = \frac{\phi_{\T_N,p}[F(\om) \alpha_z^{|\om|} ]}{\phi_{\T_N,p}[\alpha_z^{|\om|} ].}
\]
\end{lemma}

\begin{proof}
Observe that 
\begin{align*}
\sum_{\om\in\{0,1\}^{E(\T_N)}} F(\om)\left( \frac{p+z}{1-(p+z)} \right)^{|\om|} q^{k^{\bbT_N}(\om)} &= \sum_{\om\in\{0,1\}^{E(\T_N)}} F(\om)\alpha^{|\om|}_z \left( \frac{p}{1-p} \right)^{|\om|} q^{k^{\bbT_N}(\om)} \\
&= Z_{\T_N, p}\phi_{\T_N,p}[F(\om)\alpha_z^{|\om|}].
\end{align*}
The same computation with $F \equiv 1$ gives
\begin{align*}
Z_{\T_N, p+z} = Z_{\T_N, p}\phi_{\T_N,p}[\alpha_z^{|\om|}].
\end{align*}
The result follows by dividing the first display by the second one.
\end{proof}

We thus start by expanding the quantity $\phi_{\T_N,p}[ F(\om)\alpha_z^{|\om|} ] $, the goal being to realize this quantity as a weighted sum of modified polymer partition functions to which we will later apply cluster expansion. 
We start with a bit of notation.

\begin{notation}[Polymer set notation] \label{def:polymer} Introduce the following families of sets: 
\begin{itemize}
\item The set $\sfP^L_N$ was previously defined as $\{ \gamma \subset \T_N^L, \gamma \text{ connected and non-empty}\}$.
An element of $\sfP^L_N$ will be called a \emph{polymer}, and will typically be denoted by $\gamma, \bar\gamma$ or $\tilde{\gamma}$. 
We shall simply write $\{\gamma\}$ for a generic family of polymers of $\T_N^L$.
By classical bounds on lattice animals, the number of polymers on $\T_N^L$ of cardinality $n$ containing 0 is upper bounded by $\exp(c(d)n)$, where $c(d)>0$ is a constant depending only on the dimension $d$.
In what follows, $c(d)$ will always be used to refer to this constant. 


\item Define the set of \emph{independent} families of polymers as:
\[
\mathfrak{I} \coloneqq \{ \{ \gamma \}\text{ polymer family of }\sfP_N^L, \forall \gamma, \gamma'  \in \{ \gamma \}, \gamma \neq \gamma' \Rightarrow \gamma \cup \gamma' \text{ is not connected} \}. 
\]

\item For a polymer family $\{\gamma\}$, define its \emph{trace} on $\T_N^L$ by $\mathsf{Tr}(\{\gamma\}) \coloneqq  \bigcup_{\gamma \in \{\gamma\} }\gamma$.

\item For any set $A \subset \T_N^L$, define $I(A)$ as the set of families of polymers intersecting $A$, that is:
\[
I(A) \coloneqq  \{ \{\gamma\} \text{ polymer family of }\sfP_N^L, \forall \gamma \in \{\gamma\}, \gamma \cap A \neq \emptyset\}.
\]


\end{itemize}

\end{notation}

Next lemma is classical, and is proved in~\cite[Lemma 4.4]{Ott_Ising}.
For completeness, we provide a quick proof.

\begin{lemma}\label{lem:def_c(d)}
For any $\Delta \subset \T_N^L$, and any $n \geq |\Delta|$, 
\[
|\{ \{ \gamma \} \in  \mathfrak{I} \cap I(\Delta), |\mathsf{Tr}(\{\gamma\})| = n  \}| \leq 4^n\exp(c(d)n).
\]
\end{lemma}

\begin{proof}
Fix $n$ and $\Delta$ as in the statement. 
It is the case that the cardinality of $\{ \{ \gamma \} \in  \mathfrak{I} \cap I(\Delta), |\mathsf{Tr}(\{\gamma\})| = n  \}$ is roughly upper bounded by the cardinality of
\[ \{  \gamma_{1}, \dots, \gamma_{|\Delta|} \in \sfP^L_N \cup \{\emptyset\}, |\gamma_{1}| +  \dots  +  |\gamma_{|\Delta|}| = n \}. \]
Fix a family $r_1, \dots, r_{|\Delta|}$ of non-negative integers such that $r_1 + \dots + r_{|\Delta|} = n$, and recall that the number of such families is $\binom{n+|\Delta|-1}{|\Delta|-1}$. 
This reasoning implies that 
\[
|\{ \{ \gamma \} \in  \mathfrak{I} \cap I(\Delta), |\mathsf{Tr}(\{\gamma\})| = n  \}| \leq \binom{n+|\Delta|-1}{|\Delta|-1}\exp(c(d)n) \leq 4^n \exp(c(d)n),
\]  
by observing that $\binom{n+|\Delta|-1}{|\Delta|-1} \leq \binom{2n}{|\Delta|-1} \leq 2^{2n}$, since $|\Delta|\leq n$. 
\end{proof}

The quantities of interest in this paper will be expressed in terms of functions of a polymer model on $\T_N^L$. 
More precisely, we will use the so-called \emph{polymer partition function}, defined below. 

\begin{definition}[Polymer partition function] \label{def:polymer_partition_function}
Let $\sfw : \sfP^L_N \rightarrow \C$ be an \emph{activity function}. The following expression will be referred to as the \emph{polymer partition function} with activities $\sfw$:
\[
\Xi(\sfw) \coloneqq  \sum_{\{\gamma\} \in \mathfrak{I}} \prod_{\gamma\in\{\gamma\}} \sfw(\gamma).
\]
Observe that due to the non-intersection constraint, $\Xi(\sfw)$ is always a finite sum.
\end{definition}

The starting point of the analysis is the rewriting of the quantity $\phi_{\T_N,p}[F(\om) \alpha_z^{|\om|}]$ as a weighted sum of polymer partition functions, in order to use the tools provided by the cluster expansion. 
In the following lemma, the reader is invited not to focus too much on the explicit expression of the function $G_z$ and of the polymer weights $\sfw_z^{\{ \bar\gamma \}}$. 
The important aspect of the proposition is the way in which $\phi_{\T_N,p}[F(\om) \alpha_z^{|\om|}]$ is written as a combination of weighted polymer partition functions. 
The proof is rather tedious, but is nothing but a suitable resummation of the expansion obtained by decomposing over the possible realisations of the clusters $\{\calC_x\}_{x \in \bbT_N^L}$, introduced in Theorem~\ref{thm:dependency_encoding}. 
We start with some convenient notation. 
\begin{definition}
Fix $z \in \C$.
\begin{itemize}
\item For $x \in \T_N^L$ and a percolation configuration $\om$ in $\T_N$, introduce the function 
\[f^L_{z,x}(\om) \coloneqq  \big\{ \prod_{e\in E(\Lambda_L(x))} \alpha_z^{\om(e)} \big\} - 1 .\] 

\item  For any polymer $\gamma \in \sfP_N^L$, define the activity function
\[
\sfw_z(\gamma) \coloneqq  \sum_{A\subset \gamma} \P[ \prod_{x\in A} f_{z,x}(\om)\ind{\calC_A = \gamma} ].
\]

\item For any family $\{\bar\gamma\}$ of polymers, define the \emph{non-intersecting} activity function as, for $\gamma \in \sfP_N^L$:
\[
\sfw_z^{\{\bar\gamma\}}(\gamma)\coloneqq  \ind{\gamma \cap \mathsf{Tr}(\{\bar\gamma\}) = \emptyset} \sfw_z(\gamma).
\]

\item Finally, for any \emph{non-intersecting} family of polymers $\{\bar\gamma\}$, define the function:
\[
G_z(\{\bar\gamma\}) \coloneqq  \sum_{A\subset \mathsf{Tr}(\{\bar\gamma\}) } \P[F(\om) \prod_{x\in A} f_{z,x}(\om)\ind{\calC_{A \cup \Delta} = \mathsf{Tr}(\{\bar\gamma\})}].
\]
\end{itemize}

\end{definition}

\begin{lemma}\label{lem:plus_minus_one_expansion}
For any $z \in \C$ such that $\Xi(\sfw_z) \neq 0$,
\[
\phi_{\T_N, p+z}[F(\om)] = \sum_{ \{\bar\gamma\} \in \mathfrak{I} \cap I(\Delta)} G_z(\{\bar\gamma\})  \frac{\Xi (\sfw^{\{\bar\gamma\}}_z)}{\Xi(\sfw_z)},
\]
where $\Xi$ is defined in Definition~\ref{def:polymer_partition_function}.
\end{lemma}

\begin{proof}
One just has to prove that 
\[ \phi_{\T_N,p}[F(\om) \alpha_z^{|\om|}]  = \sum_{ \{\bar\gamma\} \in \mathfrak{I} \cap I(\Delta)} G_z(\{\bar\gamma\})  \Xi (\sfw^{\{\bar\gamma\}}_z). \]
Indeed, by taking $F \equiv 1$ in the above, we get
\[ \phi_{\T_N,p}[\alpha_z^{|\om|}]  = \Xi(\sfw_z),\]
so the conclusion follows directly by Lemma~\ref{lem:rewriting_phi_loc}.

We start with a classical $\pm 1$ expansion\footnote{The careful reader might note that some edges that lie on the boundary of a $L$-box are counted twice. For the identity to be correct, we need to remove some edges on the boundary in the definition of $f^L_{z,x}$; we shall always do it implicitly.}:
\[
\phi_{\T_N,p}[F(\om) \alpha_z^{|\om|}] = \phi_{\T_N,p}[F(\om) \prod_{x \in \bbT^N_L }\{ f_{z,x}(\om) + 1\}] = \phi_{\T_N,p}[F(\om) \sum_{A \subset \bbT^N_L } \prod_{x \in A } f_{z,x}(\om) ].
\]
The next step is to sum over the trace of the possible realisations of $\calC_{A \cup \Delta}$ under the measure $\P$. This is done writing:
\begin{align*}
\phi_{\T_N,p}[F(\om) \alpha_z^{|\om|}] = \sum_{\{\gamma \} \in \mathfrak{I}}
\sum_{A \subset \mathsf{Tr}(\{\gamma\})}
\P[F(\om)\prod_{x \in A}  f_{z,x}(\om)\ind{\calC_{A \cup \Delta} = \mathsf{Tr}(\{\gamma\})} ].
\end{align*}
We partition each family $\{\gamma\} \in \mathfrak{I}$ in two by taking $\{\bar\gamma\}$ to be the polymers of $\{\gamma\}$ that intersect $\Delta$ and $\{\tilde\gamma\}$ to be the polymers of $\{\gamma\}$ that do not intersect $\Delta$. 
By summing over the potential realisations of these two families instead of summing over $\{\gamma\}$, we get:
\begin{multline*}
\phi_{\T_N,p}[F(\om) \alpha_z^{|\om|}] = \sum_{\{\bar\gamma \}}
\sum_{\{\tilde\gamma \}}
\ind{\mathsf{Tr}(\{\bar{\gamma}\}) \cap \mathsf{Tr}(\{\tilde{\gamma}\}) = \emptyset}\\ 
\sum_{A \subset \mathsf{Tr}(\{\bar{\gamma}\}) \cup \mathsf{Tr}(\{\tilde{\gamma}\})}
\P[F(\om)\prod_{x \in A}  f_{z,x}(\om)\ind{\calC_{A \cup \Delta} = \mathsf{Tr}(\{\bar\gamma\}) \cup \mathsf{Tr}(\{\tilde\gamma\})} ],
\end{multline*}
where the first sum is over the families $\{\bar\gamma \} \in \mathfrak{I}\cap I(\Delta)$ such that $\Delta \subset \mathsf{Tr}(\{\bar\gamma\})$ and the second sum is over the families $\{\tilde\gamma \} \in \mathfrak{I}$. \textbf{All the sum over these symbols will always be taken as such in the remainder of the proof.}

For fixed families $\{\bar\gamma \}$ and $\{\tilde\gamma \}$ such that the indicator function is not 0 and a fixed set $A \subset \mathsf{Tr}(\{\bar{\gamma}\}) \cup \mathsf{Tr}(\{\tilde{\gamma}\})$, we can partition $A$ in two by taking 
\begin{align*}
\bar{A} &= A \cap \mathsf{Tr}(\{\bar\gamma\}) & \tilde{A} = A \cap \mathsf{Tr}(\{\tilde\gamma\}).
\end{align*}
By summing over the possible realisations of $\bar{A}$ and $\tilde{A}$ instead of summing over the possible realisations of $A$, we obtain
\begin{multline*}
\phi_{\T_N,p}[F(\om) \alpha_z^{|\om|}] = \sum_{\{\bar\gamma \}}
\sum_{\{\tilde\gamma \}}
\ind{\mathsf{Tr}(\{\bar{\gamma}\}) \cap \mathsf{Tr}(\{\tilde{\gamma}\}) = \emptyset}\\
\sum_{\bar{A} \subset \mathsf{Tr}(\{\bar{\gamma}\})}\sum_{\tilde{A} \subset \mathsf{Tr}(\{\tilde{\gamma}\})}
\P[F(\om)\prod_{x \in \bar{A}}  f_{z,x}(\om)\ind{\calC_{\bar{A} \cup \Delta} = \mathsf{Tr}(\{\bar\gamma\})} \prod_{x \in \tilde{A}}  f_{z,x}(\om)\ind{\calC_{\tilde{A}} = \mathsf{Tr}(\{\tilde\gamma\})}].
\end{multline*}

Fix $\{\bar{\gamma}\}$, $\{\tilde{\gamma}\}$, $\bar{A}$, $\tilde{A}$ as in the sums above and such that $\mathsf{Tr}(\{\bar{\gamma}\}) \cap \mathsf{Tr}(\{\tilde{\gamma}\}) = \emptyset$. 
Since the support $\Delta$ of $F$ is contained in $\mathsf{Tr}(\{\bar{\gamma}\})$, the decoupling property of $\P$ (point (3) of Theorem \ref{thm:dependency_encoding}) implies that:
\begin{multline*}
\P[F(\om)\prod_{x \in \bar{A}}  f_{z,x}(\om)\ind{\calC_{\bar{A} \cup \Delta} = \mathsf{Tr}(\{\bar\gamma\})} \prod_{x \in \tilde{A}}  f_{z,x}(\om)\ind{\calC_{\tilde{A}} = \mathsf{Tr}(\{\tilde\gamma\})}]\\
= \P[F(\om)\prod_{x \in \bar{A}}  f_{z,x}(\om)\ind{\calC_{\bar{A} \cup \Delta} = \mathsf{Tr}(\{\bar\gamma\})}] \P[\prod_{x \in \tilde{A}}  f_{z,x}(\om)\ind{\calC_{\tilde{A}} = \mathsf{Tr}(\{\tilde\gamma\})}].
\end{multline*}
Furthermore, since the elements $\tilde{\gamma}$ of $\{\tilde{\gamma}\}$ are disjoint, the decoupling property of $\P$ also implies that
\[
\P[\prod_{x \in \tilde{A}}  f_{z,x}(\om)\ind{\calC_{\tilde{A}} = \mathsf{Tr}(\{\tilde\gamma\})}] = \prod_{\tilde{\gamma} \in \{\tilde{\gamma}\}} \P[ \prod\limits_{x \in \tilde{\gamma} \cap \tilde{A}} f_{z,x}(\om)\ind{\calC_{\tilde{\gamma} \cap \tilde{A}} = \tilde{\gamma}}].
\]
Using these two observations and rearranging the terms, we get
\begin{multline*}
\phi_{\T_N,p}[F(\om) \alpha_z^{|\om|}] = \sum_{\{\bar\gamma \} }
\sum_{\bar{A} \subset \mathsf{Tr}(\{\bar{\gamma}\})}
\P[F(\om)\prod_{x \in \bar{A}}  f_{z,x}(\om)\ind{\calC_{\bar{A} \cup \Delta} = \mathsf{Tr}(\{\bar\gamma\})}]\\
\sum_{\{\tilde\gamma \}}
\ind{\mathsf{Tr}(\{\bar{\gamma}\}) \cap \mathsf{Tr}(\{\tilde{\gamma}\}) = \emptyset}
\sum_{\tilde{A} \subset \mathsf{Tr}(\{\tilde{\gamma}\})}
\prod_{\tilde{\gamma} \in \{\tilde{\gamma}\}} \P[ \prod\limits_{x \in \tilde{\gamma} \cap \tilde{A}} f_{z,x}(\om)\ind{\calC_{\tilde{\gamma} \cap \tilde{A}} = \tilde{\gamma}}].
\end{multline*}

Observe the following interchange of sum and product:
\begin{align*}
\sum_{\tilde{A} \subset \mathsf{Tr}(\{\tilde{\gamma}\})}
\prod_{\tilde{\gamma} \in \{\tilde{\gamma}\}} 
\P[ \prod\limits_{x \in \tilde{\gamma} \cap \tilde{A}} f_{z,x}(\om)\ind{\calC_{\tilde{\gamma} \cap \tilde{A}} = \tilde{\gamma}}]
= 
\prod_{\tilde{\gamma} \in \{\tilde{\gamma}\}} 
\sum_{\tilde{A} \subset \tilde{\gamma}}
\P[ \prod\limits_{x \in \tilde{A}} f_{z,x}(\om)\ind{\calC_{\tilde{A}} = \tilde{\gamma}}].
\end{align*}
Also observe that the indicator function can be written as follows:
\[
\ind{\mathsf{Tr}(\{\bar{\gamma}\}) \cap \mathsf{Tr}(\{\tilde{\gamma}\}) = \emptyset} = \prod_{\tilde\gamma \in \{\tilde\gamma\}} \ind{\tilde\gamma \cap \mathsf{Tr}(\{\bar\gamma\}) = \emptyset}.
\]
The previous two observations imply the following rewriting: 
\begin{align*}
\phi_{\T_N,p}[F(\om) \alpha_z^{|\om|}] 
= 
\sum_{\{\bar\gamma\}}
\sum_{\bar{A} \subset \mathsf{Tr}(\{\bar\gamma\} )}
\P[F(\om)\prod_{x \in \bar{A}}  f_{z,x}(\om)\ind{\calC_{\bar{A} \cup \Delta} = \mathsf{Tr}(\{\bar\gamma\})}]
\Xi(\sfw_z^{\{\bar\gamma\}}).
\end{align*}

It follows immediately from the definition of $G_z$ that
\[
\sum_{ \{\bar\gamma\} \in \mathfrak{I} \cap I(\Delta) \atop \Delta \subset \mathsf{Tr}(\{\bar\gamma\})} G_z(\{\bar\gamma\})  \Xi (\sfw_z^{\{\bar\gamma\}}).
\]
The conclusion follows by observing that $G_z(\{\bar\gamma\}) = 0$ when $\Delta$ is not contained in $\mathsf{Tr}(\{\bar\gamma\})$.
\end{proof}

\subsection{Cluster expansion bounds on the ratio of polymer partition functions}\label{subsec:clust_exp_2}

The next step estimates the ratio of the two polymer partition functions $\Xi(\sfw_z^{\{\bar\gamma\}})$ and $\Xi(\sfw_z )$ appearing in Lemma~\ref{lem:plus_minus_one_expansion}.
This will be done using the machinery of cluster expansion provided in the appendix.
To this end, we first prove that the parameter $L$ can be chosen so that the activities $\sfw_z$ satisfy the assumption of Theorem~\ref{thm:cluster_expansion}: this is the main role of the coarse-graining step induced by grouping the $\alpha_z^{|\om|}$ terms in boxes of size $L$ in the function $f^L_{z,x}$. 
This fact was checked in~\cite{Ott_RCM}, but we provide a quick proof for completeness. 

\begin{lemma}\label{lem:choosing_L}
For every $L>0$ large enough, there exists $\delta > 0$ such that for any $0 < \eps < \delta $, there exists $C_\eps > 0$ going to 0 when $\eps$ goes to 0 such that for any $|z|< \eps$ and any polymer $\gamma \in \sfP^L_N$, 
\[
|\sfw_z(\gamma)| \leq C_\eps \exp\big(-|\gamma|(2+c(d))\big).
\]
\end{lemma}
\begin{proof}
Let $\delta> 0$. 
First observe that for any fixed $L \in \Z_{> 0}$,
\begin{align}\label{eq:unif_bound_on_f}
C_\delta^L\coloneqq  \sup_{|z| < \delta}\sup_{x \in \T_N^L}\sup_{\om \in \{0,1\}^{E(\T_N^L)}} | f_{z,x}(\om)| \leq \sup_{|z| < \delta}\sup_{k \in \{0, \cdots, E(\Lambda_L)\}} |\alpha_z^k -1 | \underset{\delta\rightarrow 0}{\longrightarrow} 0. 
\end{align}

Thus, for any $|z|<\delta$, any $\gamma \in \sfP^L_N$, if $L$ and $\delta$ are such that $C_\delta^L \leq 1$, then by Lemma \ref{lem:exp_dec_cluster_sets}, we get the following bound on the modulus of the weights:
\begin{align*}
|\sfw_z(\gamma)|  
\leq
\sum_{A \subset \gamma}(C_\delta^L)^{|A|} \P[|\calC_A| \geq |\gamma|] 
&\leq 
\sum_{A\subset \gamma \atop A \neq \emptyset} (C_\delta^L)^{|A|}\exp(a_L|A| - c_L|\gamma|) \\
&\leq \sqrt{C_\delta^L} \sum_{A\subset \gamma \atop A \neq \emptyset} \left(\sqrt{C_\delta^L}\right)^{|A|}\exp(a_L|A| - c_L|\gamma|) \\
&= \sqrt{C_\delta^L} (1+\sqrt{C_\delta^L}\exp(a_L))^{|\gamma|}\exp(-c_L|\gamma|).
\end{align*}
We then choose $L$ large enough so that $c_L \geq 3+c(d)$, and $\delta$ so small that for this value of $L$, $\sqrt{C^L_\delta}\e^{a_L} \leq 1$. 
Then, it is the case that for any $|z|\leq \delta$, 
\[ |\sfw_z(\gamma)| \leq \sqrt{C_{|z|}^L} \exp(-|\gamma|(2+c(d))), \]
which concludes the proof. 
\end{proof}
\textbf{From now on, $L$ and $\delta>0$ are fixed accordingly so that both Theorem~\ref{thm:dependency_encoding} and Lemma~\ref{lem:choosing_L} hold. }
We can now estimate the ratio of the two polymer partition functions $\Xi(\sfw_z^{\{\bar\gamma\}})$ and $\Xi(\sfw_z )$ by combining Lemma~\ref{lem:choosing_L} with the cluster expansion results provided by Theorem~\ref{thm:cluster_expansion} and Corollary~\ref{thm:cluster_expansion_inter_A}. 

\begin{lemma}\label{lem:ratio_polymers}
For any $0 < \eps < \delta$, there exists $c_\eps > 0$ going to 0 when $\eps$ tends to 0 such that for any $|z|< \eps$ and any family $\{ \bar\gamma\} \in \mathfrak{I}$, 
\[
\left| \frac{\Xi(\sfw_z^{\{\bar\gamma\}})}{\Xi(\sfw_z)} \right| \leq \exp(c_\eps |\mathsf{Tr}(\{\bar\gamma\})| ). 
\]
\end{lemma}

\begin{proof}
Fix $\{ \bar\gamma\} \in \mathfrak{I}$.
First observe that for any polymer $\gamma$, $|\sfw_z^{\{\bar\gamma\}}(\gamma)| \leq |\sfw_z(\gamma)|$. 
Thus, a convergent cluster expansion for $\Xi(\sfw_{z})$ implies a convergent cluster expansion for $\Xi(\sfw_z^{\{\bar\gamma\}})$.
As both $L$ and $\delta$ are given by Lemma~\ref{lem:choosing_L}, both $\Xi(\sfw_z)$ and $\Xi(\sfw_z^{\{\bar\gamma\}})$ admit a convergent cluster expansion due to Theorem~\ref{thm:cluster_expansion}. 

Further observe that for a polymer $\gamma$ that does not intersect the trace $\mathsf{Tr}(\{\bar\gamma\})$, one has $\sfw_z^{\{\bar\gamma\}}(\gamma) = \sfw_z(\gamma)$. 
In the converse case, we simply have $\sfw_z^{\{\bar\gamma\}}(\gamma) = 0$. 
This implies a simplification in the ratio of exponentials appearing in the cluster expanded forms of the polymer partition functions. Namely, those observations imply that: 
\[
\frac{ \Xi(\sfw_z^{\{\bar\gamma\}})}{\Xi(\sfw_z)}  \\ =  \exp\big(- \sum_{n \geq 0}\sum_{\gamma_1, \cdots, \gamma_n \in \sfP^L_N \atop \mathsf{Tr}(\{\bar\gamma\}) \cap (\gamma_1\cup\cdots\cup\gamma_n) \neq \emptyset}\varphi_n(\gamma_1, \cdots, \gamma_n) \prod_{k=1}^n \sfw_z(\gamma_k)\big).
\]
We now use Corollary~\ref{thm:cluster_expansion_inter_A} together with Lemma~\ref{lem:choosing_L} to conclude that there exists $c_\eps > 0$ going to 0 when $\eps$ tends to 0 such that, for any $|z|< \eps < \delta$,
\[
\left| \frac{\Xi(\sfw_z^{\{\bar\gamma\}})}{\Xi(\sfw_z)} \right| \leq \exp(\tilde{c}_\eps |\mathsf{Tr}(\{\bar\gamma\}| ).
\]
\end{proof}

\subsection{Exponential summability of $G_z$}\label{subsec:exp_summability}

We turn to the estimation of the function $G_z$. 
In fact, we will prove that it is ``exponentially summable'', in the following sense. 

\begin{lemma}\label{lem:bound_G}
There exists $c_\eps>0$, going to 0 when $\eps$ tends to 0, such that for every $\eps > 0$ small enough, if $|z| < \eps$, then:
\begin{align*}
\sum_{\{\bar\gamma\} \in \mathfrak{I}\cap I(\Delta)} |G_z(\{\bar\gamma\})|\exp(-c_\eps |\mathsf{Tr}(\{ \bar\gamma \})|) \leq \phi_{\T_N,p}[F(\om)].
\end{align*}
\end{lemma}

\begin{proof}
Let $\eps < \delta$, and set $\tilde p = p + \eps$. Also write $\mathbf{1}$ for the percolation configuration on $\T_N$ consisting of open edges only.
The proof consists in expanding the quantity $\phi_{\T_N,p}[F(\om)\alpha_{\eps}^{|\mathbf{1}|}]$ in the cluster expansion formalism to provide an exponential upper bound on $G_z$. 
Indeed, observe that\footnote{This statement is similar to that of Lemma~\ref{lem:plus_minus_one_expansion}; the reader can check that the proof applies \emph{mutatis mutandis}.}
\begin{equation}\label{equ:rewriting_F_proof_exp_sum_G}
\phi_{\T_N,p}[F(\om)] =  \frac{\phi_{\T_N,p}[F(\om)\alpha_{\eps}^{|\mathbf{1}|}]}{\phi_{\T_N,p}[\alpha_{\eps}^{|\mathbf{1}|}]} = \sum_{\{\bar\gamma\} \in \mathfrak{I}\cap I(\Delta)} G^{\mathbf{1}}_{\tilde p}(\{\bar\gamma\})\frac{\Xi(\sfw_{\tilde p}^{\mathbf{1}, \{\bar\gamma\}})}{\Xi(\sfw_{\tilde p}^{\mathbf{1}})},
\end{equation}
where 
\begin{itemize}
\item  For any polymer $\gamma \in \sfP_N^L$, the activity function is modified as follows:
\[
\sfw^{\mathbf{1}}_{\tilde p}(\gamma) \coloneqq  \sum_{A\subset \gamma} \P[ \prod_{x\in A} f_{\tilde p,x}(\mathbf{1})\ind{\calC_A = \gamma} ].
\]

\item For any family $\{\bar\gamma\}$ of polymers, the \emph{non-intersecting} activity function is modified as follows:
\[
\sfw^{\mathbf{1}, \{\bar\gamma\}}_{\tilde p}(\gamma)\coloneqq  \ind{\gamma \cap \mathsf{Tr}(\{\bar\gamma\}) = \emptyset} \sfw^{\mathbf{1}}_{\tilde p}(\gamma).
\]

\item Finally the function $G_z$ is modified as follows
\[
G_{\tilde p}^{\mathbf{1}}(\{\bar\gamma\}) \coloneqq  \sum_{A\subset \mathsf{Tr}(\{\bar\gamma\}) } \P[F(\om) \prod_{x\in A} f_{\tilde p,x}(\mathbf{1})\ind{\calC_{A \cup \Delta} = \mathsf{Tr}(\{\bar\gamma\})}].
\]

\end{itemize}

Observe that $\sfw^{\mathbf{1}}_{\tilde p}, \sfw^{\mathbf{1}, \{\bar\gamma\}}_{\tilde p}$ and $G_{\tilde p}^{\mathbf{1}}$ are real-valued functions taking only non-negative values.
Indeed, when $\eps > 0$, it is the case that $\alpha_{\eps} > 1$, so that $f_{\tilde p, x}(\mathbf{1}) = \alpha_{\eps}^{(2L)^d} - 1 > 0$. 
The fact that $G_{\tilde p}^{\mathbf{1}}$ is non-negative follows from the previous observation together with the positivity of~$F$. 

Up to decreasing the value of $\eps < \delta$, for any family $\{\bar\gamma\} \in \mathfrak{I}\cap I(\Delta)$ and any $z \in \C$ with $|z| < \eps/2$, one has $|G_z(\{\bar\gamma\})| \leq  G_{\tilde p}^{\mathbf{1}}(\{\bar\gamma\})$.
This can be seen by observing that $F \geq 0$, and that for any percolation configuration $\om \in \{0,1\}^{E(\T_N)}$, then $0 \leq |f_{z,x}(\om)| \leq f_{\tilde p, x}(\mathbf{1})$ when $|z| < \eps/2$.


Now, observe that $f_{\tilde p, x}(\mathbf{1}) < C_\eps^L \leq C_\delta^L$, where $C_\delta^L$ was defined in~\eqref{eq:unif_bound_on_f}.
Thus, the proof of the bound of Lemma~\ref{lem:choosing_L} can be repeated for $\sfw_{\tilde p}^{\mathbf{1}}$. So, by Theorem~\ref{thm:cluster_expansion}, we get convergent cluster expansions for the two functions $\Xi(\sfw_{\tilde p}^{\mathbf{1}})$ and $\Xi(\sfw_{\tilde p}^{\mathbf{1}, \{\bar\gamma\}})$. By the same reasoning as in the proof of Lemma~\ref{lem:ratio_polymers}, it follows that:

\[
\frac{\Xi(\sfw_{\tilde p}^{\mathbf{1}, \{\bar\gamma\}})}{\Xi(\sfw_{\tilde p}^{\mathbf{1}})} = \exp\big(-\sum_{n\geq 0}\sum_{\gamma_1, \dots, \gamma_n \in \sfP^L_N \atop \mathsf{Tr}(\{\bar\gamma\}) \cap (\gamma_1 \cup \cdots \cup \gamma_n) \neq \emptyset }\varphi_n(\gamma_1, \dots, \gamma_n) \prod_{k=1}^n \sfw_{\tilde p}^{\mathbf{1}}(\gamma_k) \big).
 \]

By\footnote{Note that this time, our purpose is to show that the ratio of the partition function is not smaller than an exponential decaying as slow as desired. Thus, we use the lower bound provided by the second item of Corollary~\ref{thm:cluster_expansion_inter_A} and not the upper bound given by Lemma~\ref{lem:ratio_polymers}.} Corollary~\ref{thm:cluster_expansion_inter_A}, there exists $\tilde{c}_\eps > 0 $, going to 0 when $\eps$ goes to 0 such that:

\[
 \frac{\Xi(\sfw_{\tilde p}^{\mathbf{1}, \{\bar\gamma\}})}{\Xi(\sfw_{\tilde p}^{\mathbf{1}})} \geq \exp(-\tilde{c}_\eps|\mathsf{Tr}(\{\bar\gamma\})|).
\]

Going back to~\eqref{equ:rewriting_F_proof_exp_sum_G}, we obtain that, when $|z| < \eps/2$,
\begin{align*}
\phi_{\T_N,p}[F(\om)]  =  \sum_{\{\bar\gamma\} \in \mathfrak{I}\cap I(\Delta)} G^{\mathbf{1}}_{\tilde p}(\{\bar\gamma\})\frac{\Xi(\sfw_{\tilde p}^{\mathbf{1}, \{\bar\gamma\}})}{\Xi(\sfw_{\tilde p}^{\mathbf{1}})} 
&\geq  \sum_{\{\bar\gamma\} \in \mathfrak{I}\cap I(\Delta)} G^{\mathbf{1}}_{\tilde p}(\{\bar\gamma\})\exp(-\tilde{c}_\eps |\mathsf{Tr}(\{\bar\gamma \})|) \\
&\geq  \sum_{\{\bar\gamma\} \in \mathfrak{I}\cap I(\Delta)} |G_{z}(\{\bar\gamma\})|\exp(-\tilde{c}_\eps |\mathsf{Tr}(\{\bar\gamma \})|).
\end{align*}
This concludes the proof. 
\end{proof}

\textbf{In what follows, we will set $c_\eps$ to be the maximal value of the constants $c_\eps$ and $\tilde{c}_\eps$ given respectively by Lemmas~\ref{lem:ratio_polymers} and~\ref{lem:bound_G}.}

\subsection{Proof of Theorem~\ref{thm:local_analyticity}}\label{sec:proof_thm_local}

Lemmas~\ref{lem:ratio_polymers} and~\ref{lem:bound_G} will be our main ingredients for providing a proof of Theorem~\ref{thm:local_analyticity}. 
We first write the main result of the section, and shall explain how it easily implies the result of Theorem~\ref{thm:local_analyticity}. 
Recall that $L$ and $\delta > 0$ are chosen such that Lemmas~\ref{lem:ratio_polymers} and~\ref{lem:bound_G} are valid. 

\begin{proposition}\label{prop:unif_bound_ratio_probas}
There exist $\delta > 0$ and a sequence $c_\eps$ going to 0 when $\eps$ goes to 0 such that for any $|z|\leq \eps < \delta$, any $N \geq 0$, and any non-negative local function $F$ which is not identically 0, 
\begin{equation}\label{equ:target_bound_thm_local_analyticity}
\left| \frac{\phi_{\T_N, p+z}[F(\om)]}{\phi_{\T_N, p}[F(\om)]} \right| \leq 2\exp(c_\eps |\Delta| ).
\end{equation}
\end{proposition}


We first explain how Proposition~\ref{prop:unif_bound_ratio_probas} implies the result of Theorem~\ref{thm:local_analyticity}.

\begin{proof}[Proof of Theorem~\ref{thm:local_analyticity}]
Consider a non-negative local function $F$ that is not identically 0.
Observe that for any $N \in \Z_{> 0}$, the function \[ F_N : z \in \D_\delta(0) \mapsto \phi_{\T_N, p+z}[F(\om)] \in \C \] is well-defined and analytic. Indeed, Proposition~\ref{prop:unif_bound_ratio_probas} ensures in particular that the partition function renormalizing this quantity does not vanish in $\D_\delta(0)$. 
We will use Vitali's convergence theorem (see~\cite[Theorem 4, p.~164]{remmert2013classical}) to conclude and start by checking the validity of its assumptions.
By existence of the infinite volume measure\footnote{This existence is formally guaranteed by the differentiability of the free energy at $p$, which is a consequence of the result of~\cite{Ott_RCM}.  
To apply Vitali's theorem, one may alternatively observe that the interval $(-\delta, \delta)$ contains a dense subset of parameters at which the infinite-volume limit measure is unique due to~\cite[Theorem 4.63]{grimmett_RCM}.} in a neighbourhood of $p$ when $(d,p,q) \in \calG_{\mathsf{FK}}$, $F_N$ does converge on the real interval $(-\delta, \delta)$. Observe that this interval has an accumulation point on $\D_\delta(0)$, and that $F_N$ is uniformly bounded on $\D_\delta(0)$ by Proposition \ref{prop:unif_bound_ratio_probas}. 
It then follows by Vitali's convergence theorem that $F_N$ converges uniformly toward an analytic function on $\D_\delta(0)$, which proves item (1) of Theorem~\ref{thm:local_analyticity}. 
Point (2) now follows by taking $N$ to infinity in Proposition~\ref{prop:unif_bound_ratio_probas}.
This concludes the proof. 
\end{proof}

We now turn to the proof of the important Proposition~\ref{prop:unif_bound_ratio_probas}. 

\begin{proof}[Proof of Proposition~\ref{prop:unif_bound_ratio_probas}]

Let $F$ be a function that is non-negative, local and not identically 0. 
The target bound~\eqref{equ:target_bound_thm_local_analyticity} is clearly invariant by scaling, so we may assume that $\Vert F \Vert_\infty = 1$.

Fix $\eps < \delta$. 
The value of $L$ (resp. $\delta$) will potentially have to be increased (resp. decreased); we will mention it when it is the case. 
Consider $|z| < \eps$. 
Since the choice of the parameters implies the convergence of the cluster expansion of $\Xi(\sfw_z)$, it is the case that $\Xi(\sfw_z) \neq 0$, so that Lemma~\ref{lem:plus_minus_one_expansion} implies that $\phi_{\T_N, p+z}[F(\om)]$ is well defined and satisfies
\[
\frac{\phi_{\T_N, p+z}[F(\om)]}{\phi_{\T_N, p}[F(\om)]}  = \frac{1}{\phi_{\T_N,p}[F(\om)]}\sum_{\{\bar\gamma\}\in \mathfrak{I}\cap I(\Delta)} G_z(\{\bar\gamma\}) \frac{\Xi(\sfw_z^{\{\bar\gamma\}})}{\Xi(\sfw_z)}.
\]
We fix a constant $K > 0$, the value of which will be carefully set at the end of the proof, and split the sum as follows:
\begin{multline*}
\sum_{\{\bar\gamma\}\in \mathfrak{I} \cap I(\Delta)} G_z(\{\bar\gamma\})\frac{\Xi(\sfw_z^{\{\bar\gamma\}})}{\Xi(\sfw_z)} = 
\\
\underbrace{\sum_{\{\bar\gamma\}\in \mathfrak{I}\cap I(\Delta) \atop |\mathsf{Tr}(\{\bar\gamma\})| \leq K|\Delta|} G_z(\{\bar\gamma\})\frac{\Xi(\sfw_z^{\{\bar\gamma\}})}{\Xi(\sfw_z)} }_{S^1_K}+ \underbrace{\sum_{\{\bar\gamma\}\in \mathfrak{I}\cap I(\Delta) \atop |\mathsf{Tr}(\{\bar\gamma\})| > K|\Delta| } G_z(\{\bar\gamma\})\frac{\Xi(\sfw_z^{\{\bar\gamma\}})}{\Xi(\sfw_z)}}_{S_K^2}.
\end{multline*}
We will bound $S_K^1$ and $S_K^2$ separately, and it will be observed that the main contribution comes from the term $S_K^1$. 

\medskip
\noindent\textbf{Bound on $S_K^1$.}
Observe that 
\begin{align*}
|S_K^1| &\leq \sum_{\{\bar\gamma\}\in \mathfrak{I}\cap I(\Delta) \atop |\mathsf{Tr}(\{\bar\gamma\})| \leq K|\Delta|} |G_z(\{\bar\gamma\})| \Big|\frac{\Xi(\sfw_z^{\{\bar\gamma\}})}{\Xi(\sfw_z)} \Big| \\
&\leq \sum_{\{\bar\gamma\}\in \mathfrak{I}\cap I(\Delta) \atop |\mathsf{Tr}(\{\bar\gamma\})| \leq K|\Delta|} |G_z(\{\bar\gamma\})| \exp(c_\eps |\mathsf{Tr}(\{\bar\gamma\}) | ) \\
&\leq \exp(2c_\eps K|\Delta| )\sum_{\{\bar\gamma\}\in \mathfrak{I}\cap I(\Delta) \atop |\mathsf{Tr}(\{\bar\gamma\})| \leq K|\Delta|} |G_z(\{\bar\gamma\})| \exp(- c_\eps |\mathsf{Tr}(\{\bar\gamma\}) | ) \\
&\leq \exp(2c_\eps K|\Delta| ) \phi_{\T_N,p}[F(\om)].
\end{align*}
In the second line, we used Lemma~\ref{lem:ratio_polymers}. 
In the third line we used that $|\mathsf{Tr}(\{\bar\gamma\})| \leq K|\Delta|$ due to the definition of $S^1_K$.
Finally we used Lemma~\ref{lem:bound_G} in the last line (the fact that the sum is truncated even improves the bound provided by Lemma~\ref{lem:bound_G} as this sum has positive terms). 
Observe that this bound is valid for any value of the constant $K$.

We thus proved that 
\begin{equation}\label{equ:bound_S1}
|S^1_K|\leq\exp(2c_\eps K|\Delta| )\phi_{\T_N,p}[F(\om)] .
\end{equation}

\medskip
\noindent\textbf{Bound on $S_K^2$.}
The bound we get is more combinatorial in essence, and in particular does not use the control on $G_z(\{\bar\gamma\})$ provided by Lemma~\ref{lem:bound_G}.  
By definition of $G_z(\{\bar\gamma\})$, $S^2_K$ can be expanded under the following form:
\begin{align*}
S^2_K =
 \sum_{\{\bar\gamma\}\in \mathfrak{I}\cap I(\Delta) \atop |\mathsf{Tr}(\{\bar\gamma\})| > K|\Delta|}
 \sum_{A\subset \mathsf{Tr}(\{\bar\gamma\}) }
 \P[F(\om) \prod_{x\in A} f_{z,x}(\om)\ind{\calC_{A \cup \Delta} = \mathsf{Tr}(\{\bar\gamma\})}]
 \frac{\Xi(\sfw_z^{\{\bar\gamma\}})}{\Xi(\sfw_z)}. 
\end{align*}

As previously we use the triangular inequality to control this sum, and the ratio $\vert \frac{\Xi(\sfw_z^{\{\bar\gamma\}})}{\Xi(\sfw_z)} \vert$ is controlled by Lemma~\ref{lem:ratio_polymers}. 
Now, observe that due to the definition of $C_\eps^L$ provided by~\eqref{eq:unif_bound_on_f}, we have 
\begin{align*}
|S^2_K| 
\leq
\sum_{\{\bar\gamma\}\in \mathfrak{I}\cap I(\Delta) \atop |\mathsf{Tr}(\{\bar\gamma\})| > K|\Delta|}
\exp(c_\eps |\mathsf{Tr}(\{\bar\gamma\})|)
\sum_{A\subset \mathsf{Tr}(\{\bar\gamma\}) }  (C^L_\eps)^{|A|}
\P[F(\om) \ind{\calC_{A \cup \Delta} = \mathsf{Tr}(\{\bar\gamma\})}].
\end{align*}
Since $\Vert F \Vert_{\infty} = 1$, we then get
\[
|S^2_K| 
\leq 
\sum_{\{\bar\gamma\}\in \mathfrak{I}\cap I(\Delta) \atop |\mathsf{Tr}(\{\bar\gamma\})| > K|\Delta|}
\exp(c_\eps |\mathsf{Tr}(\{\bar\gamma\})|)
\sum_{A\subset \mathsf{Tr}(\{\bar\gamma\}) }  (C^L_\eps)^{|A|}
\P[|\calC_{A \cup \Delta}| \geq |\mathsf{Tr}(\{\bar\gamma\})|].
\]

By Lemma~\ref{lem:exp_dec_cluster_sets}, 
\[\P[|\calC_{A \cup \Delta}| \geq |\mathsf{Tr}(\{\bar\gamma\})|] \leq \exp(- c_L |\mathsf{Tr}(\{\bar\gamma\})| + a_L(|\Delta|+|A|)).\] 
We first sum this estimate over $A\subset \mathsf{Tr}(\{\bar\gamma\})$ and use the binomial formula to obtain: 

\[
|S^2_K| 
\leq 
\sum_{\{\bar\gamma\}\in \mathfrak{I}\cap I(\Delta) \atop |\mathsf{Tr}(\{\bar\gamma\})| > K|\Delta|}
\exp\Big( \big[ c_\eps - c_L+ \log(1+C_\eps^L\e^{a_L}) \big]   |\mathsf{Tr}(\{\bar\gamma\})| \Big)
\exp(a_L |\Delta|).
\]

At this stage, it is useful to use the finite-energy property\footnote{There exists a configuration $\om_0$ in $\{0,1\}^{\Delta}$ such that for any percolation configuration $\om$ on $\T_N$ coinciding with $\om_0$ on $\Delta$, then $F(\om) = 1$. Thus, $\phi_{\T_N,p}[F(\om)] \geq \phi_{\T_N,p}[ \om \overset{\Delta}{\equiv} \om_0 ] \geq \exp(-c_{\mathsf{FE}}|\Delta|)$ by a classical finite energy argument for the measure $\phi_p$ with $p \in (0,1)$. } of the measure $\phi_{p}$, and to introduce a constant $c_{\mathsf{FE}} > 0$, that depends only on $p$ and $q$, such that
\[  \phi_{\T_N,p}[F(\om)] > \exp(- c_{\mathsf{FE}} |\Delta|). \] 
That way, one has that

\begin{multline*}
(\phi_{\T_N,p}[F(\om)] )^{-1}
|S^2_K| 
\leq \\
\sum_{\{\bar\gamma\}\in \mathfrak{I}\cap I(\Delta) \atop |\mathsf{Tr}(\{\bar\gamma\})| > K|\Delta|}
\exp\Big( \big[ c_\eps - c_L+ \log(1+C_\eps^L\e^{a_L}) \big]   |\mathsf{Tr}(\{\bar\gamma\})| \Big)
\exp((a_L + c_{\mathsf{FE}}) |\Delta|).
\end{multline*}

We further continue to bound this sum, by noticing the fact that $|\Delta| < \frac{1}{K}|\mathsf{Tr}(\{\bar\gamma\})|$ to deal with the term $(a_L+c_{\mathsf{FE}})|\Delta| <(\frac{a_L}{K} + \frac{c_{\mathsf{FE}}}{K})|\mathsf{Tr}(\{\bar\gamma\}) |$. 
We then resum over the possible values of $|\mathsf{Tr}(\{\bar\gamma\})|$, and use the combinatorial Lemma~\ref{lem:def_c(d)}, which bounds the number of independent families of polymers intersecting $\Delta$ and having trace of cardinality $n$, by $\exp((c(d)+\log(4))n)$. 
We finally obtain: 

\[
(\phi_{\T_N,p}[F(\om)] )^{-1}
| S^2_K |
\leq
\sum_{n > K|\Delta|}
\exp\big(n( c_\eps - c_L + \log(1+C_\eps^L\e^{a_L} ) + \tfrac{a_L + c_{\mathsf{FE}}}{K} + c(d) + \log(4) )\big). 
\]
 
Up to potentially increasing the value of $L$, assume that it is chosen so large that $c_L > c(d) + \log(4) + 2$. 
Now, the fact that both $C_\eps^L, c_\eps \rightarrow 0$ as $\eps \rightarrow 0$ implies that one can choose $\delta > 0$ so small that for any $\eps < \delta,  c_\eps + \log(1+C_\eps^L\e^{a_L} ) < 1/2$.
Finally, select $K> 0$ so large that  $\frac{a_L+c_{\mathsf{FE}}}{K} < 1/2$.
All in all, we obtain that:

\[
(\phi_{\T_N,p}[F(\om)] )^{-1} | S^2_K |
\leq
\sum_{n > K|\Delta|} \exp(-n) \leq 1 \leq \exp(2c_\eps K|\Delta|) .
 \]

Thus, for this choice of $L, \eps$, and $K$, we obtain that:  

\begin{equation}\label{equ:bound_S^2}
 |S^2_K|\leq \exp(2c_\eps K|\Delta|)\phi_{\T_N,p}[F(\om)] .
 \end{equation}

\medskip
\noindent\textbf{Conclusion.}
Consider the choice of $L, \eps < \delta$ and $K$ made at the previous paragraph. 
Combining~\eqref{equ:bound_S1} and~\eqref{equ:bound_S^2} directly implies that: 
\[
\left\vert  \frac{\phi_{\T_N, p+z}[F(\om)] }{\phi_{\T_N, p}[F(\om)]}\right\vert \leq \frac{|S^1_K| + |S^2_K|}{\phi_{\T_N,p}[F(\om)]} \leq 2 \exp(2c_\eps K|\Delta| ).
\]

As the values of $L,K$ and $\eps < \delta_0$ are fixed independently of $N$, and as $c_\eps$ tends to 0 when $\eps$ tends to 0, the proof of the proposition is completed by altering the value of $c_\eps$.
\end{proof}

 \section{Analyticity of the susceptibility and of the magnetisation}\label{sec:mag_susc}

We derive from the previous section the analyticity of $\chi$ and of $\theta$ in the suitable range of parameters. 
As we shall see, once the item (2) of Theorem~\ref{thm:local_analyticity} has been established, those two properties fall under the scope of the method developed in~\cite{PG_analyticity} in the framework of Bernoulli percolation.

\subsection{Analyticity of $\chi$ in the subcritical regime}

As an easy example, we start by proving that the susceptibility of the random-cluster model is analytic in the subcritical regime, giving the proof of Theorem~\ref{thm:analyticity_susceptibility}. 
We invite the reader unfamiliar with the work of~\cite{PG_analyticity} to have a look at its proof to understand the importance of the item (2) of Theorem~\ref{thm:local_analyticity} in the analysis. 

\begin{proof}
Let $p_0<p_c$, and observe that $(d,p_0,q) \in \calG_{\mathsf{FK}}$ for this choice of parameters. We denote by $\calC_0$ the connected component of 0. As $p_0<p_c$, it follows from~\cite{DCRT_sharpness} and a classical renormalisation procedure that the volume of $\calC_0$ has exponential tails under the random-cluster measure.
Thus, we may choose $\eps > 0$ so small that 
\[ \sum_{n\geq1} n\e^{(2d) c_\eps n}\phi_{p_0,q}[|\calC_0| = n] < \infty, \]
where $c_\eps$ was constructed in Theorem~\ref{thm:local_analyticity}.

Now, observe that for all $p \in (0,p_c)$,
\begin{align*}
\chi_{\mathsf{FK}}(p) = \phi_{p,q}[|\calC_0|] =  \sum_{n \geq 1} n \sum_{C \ni 0 \atop {|C| = n \atop C \text{ connected}}} \phi_{p,q}[\ind{\calC_0 = C}(\om)].  
\end{align*}

For any finite connected set $C$ containing 0, observe that the function $\ind{\calC_0 = C}$ is local. 
Possibly after reducing $\eps$, we may thus consider the analytic extension $\phi_{p_0+z,q}[\ind{\calC_0 = C}(\om)]$ constructed in Theorem~\ref{thm:local_analyticity}, and formally define for $z \in \D_\eps(0)$
\begin{align}\label{eq:def_extension_chi}
\chi_{\mathsf{FK}}(p_0+z) \coloneqq  \sum_{n \geq 1} n \sum_{C \ni 0 \atop {|C| = n \atop C \text{ connected}}} \phi_{p_0+z,q}[\ind{\calC_0 = C}(\om)]. 
\end{align} 
This series of analytic functions defines an analytic function provided that it converges uniformly on $\D_\eps(0)$ \cite[Chapter 5.1.1]{Ahlfors_comp_analysis}. Observe that if $C$ has volume $n$, the support of $\ind{\calC_0 = C}$ has volume at most $2dn$. 
Item (2) of Theorem~\ref{thm:local_analyticity} can thus be invoked to obtain:
\[  n  \sum_{C \ni 0 \atop {|C| = n \atop C \text{ connected}}} \big\vert \phi_{p_0+z,q}[\ind{\calC_0 = C}(\om)] \big\vert \leq  n\exp(2d c_\eps n)\phi_{p_0,q}[|\calC_0| = n]. \]
By choice of $\eps$, we get the following uniform bound
\begin{align}
	\sum_{n \geq 1} n \sup\limits_{|z| < \eps} \sum_{C \ni 0 \atop {|C| = n \atop C \text{ connected}}} \big| \phi_{p_0+z,q}[\ind{\calC_0 = C}(\om)] \big|  < \infty.
\end{align}
Hence, by the Weierstrass M-test, the series given by \eqref{eq:def_extension_chi} is uniformly convergent on $\D_\eps(0)$. It follows that this series defines an analytic extension of $p \mapsto \chi_{\mathsf{FK}}(p)$ around $p_0$. 

The result for $\chi_{\mathsf{Potts}}$ follows by the identity $\chi_{\mathsf{Potts}}(\beta) = \tfrac{q-1}{q}\chi_{\mathsf{FK}}(p(\beta))$ for any $\beta < \beta_c$, where we recall that $p(\beta) = 1- \e^{-\beta}$. 
This identity is an easy consequence of the Edwards--Sokal coupling. 
\end{proof}

\begin{remark}
In the case of the Ising model, a simpler proof can be derived using~\cite{Ott_Ising}, and observing that in the subcritical regime, the Ising susceptibility at $\beta$ is the second partial derivative in $h$ of the free energy evaluated at $(\beta, h=0)$. 
As~\cite{Ott_Ising} shows the joint analyticity of this quantity in $(\beta,h)$ when $\beta<\beta_c$, this gives a proof of the analyticity of the susceptibility in the case of the Ising model. 
\end{remark}

\subsection{Analyticity of $\theta$}

The previous example stresses the importance of the exponential upper bound given by Theorem~\ref{thm:local_analyticity} in order to prove the analyticity of percolation quantities that can be expressed as a sum of probabilities of local events. Unfortunately, a naive adaptation of the proof of Proposition~\ref{thm:analyticity_susceptibility} is doomed to fail for proving the analyticity of $\theta$ in the supercritical regime, as the quantity $\phi_{p,q}[|\calC_0| = n]$ is expected to decay as $\exp(-cn^{\frac{d-1}{d}})$ (see~\cite{CCN_sharpness} for the ``historical'' proof in the case of Bernoulli percolation). Indeed, this bound is not good enough for the above-mentioned strategy to work, because one would require an exponential bound.

In~\cite[Section 8]{PG_analyticity}, the authors fix this issue by introducing a suitable notion of \textit{separating components} that separate 0 from infinity, and by showing that the probability that the percolation configuration exhibits some \textit{separating component} of size $n$ decays exponentially in $n$ after a suitable coarse-graining. This gives an expression of $\theta$ as a sum of probabilities of local events with the required exponential bound. 

The argument of~\cite{PG_analyticity} is mostly based on combinatorial arguments, which do not rely on the independence of the edges in Bernoulli percolation. 
In the proof below, we give the details on how to extend their argument to the general FK-percolation case by using Theorem~\ref{thm:local_analyticity}.
As most of the argument is the same as~\cite{PG_analyticity}, the proof is not self-contained, and we point out to the relevant statements of~\cite{PG_analyticity} that we use.

\begin{proof}[Sketch of proof of Theorem \ref{thm:analyticity_magnetisation}]
Let us take $(d,p_0,q) \in \calG_{\mathsf{FK}}$. If $p_0 < p_c$ then $\theta(p_0)$ is $0$ in a neighbourhood of $p_0$, so the result for $\theta$ is obviously true in this case. So, let us suppose that $p_0 > p_c$. 
Before starting with the argument, let us recall that in the regime under consideration, the model possesses the \textit{local uniqueness} property (stated in point (2.b) of Definition \ref{def:good_parameters}). 
In the proof, this property will be used in a classical way to perform a coarse-graining argument. 
We fix a positive integer $N$ and we call $\bbLN$ the coarse-grained lattice with mesh $N$, endowed with the $*$-connectivity.
This is the lattice whose vertex set is given by $N \bbZ^d$ and where two points are connected if and only if they are at $L^\infty$ distance equal to $N$ (with $\bbZ^d$ seen as embedded in $\bbR^d$). 

For a given configuration $\om$ on $\bbZ^d$,~\cite{PG_analyticity} introduces the notion of \emph{separating component}.
Roughly, a separating component corresponds to a connected component of bad boxes, together with their boundary which consists of good boxes,
where a good (resp. bad) box is a box in which some \textit{local uniqueness} event occurs (resp. does not occur). 
The precise definition of a separating component is not needed here, but we refer the reader to section 8.2 of~\cite{PG_analyticity}. 
We will just use three claims about \textit{separating components}, stated below. 
These claims are proven in~\cite{PG_analyticity} for Bernoulli percolation and we explain how to prove them in the general FK-percolation case.\\


The first claim is a simple geometric observation that follows from the definition of a separating component.
In what follows, when $S$ is a connected set of $\bbLN$, we say that $S$ \emph{occurs} if $S$ is a separating component for $\om$. \\

\textbf{Claim A.} The event that $S$ occurs is a local event with support of size $C N^d |S|$ (where $C$ is a constant that only depends on the dimension).\\

In order to state the next claim, we define $MS^N_n$ as the set of finite collections of pairwise disjoint $\bbLN$-connected sets $\{S_1,S_2,\dots,S_k\}$ such that $|S_1| + |S_2| + \dots + |S_k| = n$. With a slight abuse of notation we still write $S$ for a generic element of $MS^N_n$. We say that $S$ occurs in a configuration $\om$ if all the elements of $S$ are separating components in $\om$. \\

\textbf{Claim B.} There exists $N = N(p_0) > 0$, $t = t(p_0) > 0$ and an interval $(a,b)$ containing $p_0$ such that
\begin{align*}
\sum\limits_{S \in MS^N_n} \phi_{p,q}[ S \mbox{ occurs}] \leq e^{-tn},
\end{align*}
for every $n \geq 1$ and every $p \in (a,b)$.\\

The careful reader may check that Claim B can be obtained by a simple adaptation of the proof of \cite[Lemma 8.6]{PG_analyticity}. 
Indeed, the only change required is to use the uniformity in boundary conditions of the \textit{local uniqueness} property (2.b) in place of the independence property of Bernoulli percolation.
Claim B then follows by a standard Peierls' argument. \\

\textbf{Claim C.} Let $N$ and $(a,b)$ be as in the previous claim. We denote by $D_N$ the event that the cluster of $0$ has a diameter smaller than $N/5$. Then it holds that
\begin{align*}
1 - \theta(p) = \phi_{p,q}[D_N] - \phi_{p,q}[\text{Some separating component occurs}, D_N^c],
\end{align*}
for every $p \in (a,b)$.\\

The claim above is stated in \cite[equation (8.5)]{PG_analyticity} as a consequence of \cite[Lemmas 8.2 and 8.4]{PG_analyticity} whose proof only rely on graph-theoretic arguments and on Claim B. 
Thus, Claim C is also valid in the general FK-percolation case. 
Let us now combine the claims to obtain an analytic extension of $\theta$. 
Thus, choose $N$ as fixed by Claim B. 
We use the same resummation as in the Bernoulli case, but Theorem \ref{thm:local_analyticity} is now required to prove the analyticity of the resummation of $\theta$.

Since two intersecting separating components cannot occur at the same time, it follows from Claim B and the inclusion-exclusion principle that 
\begin{align}\label{eq:extension_theta_sep_comp}
	\phi_{p,q}[\mbox{Some separating component} &\mbox{ occurs}, D_N^c] \\ 
	&= \sum\limits_{n \geq 1} \sum\limits_{S \in MS_n^N} (-1)^{c(S)+1} \phi_{p,q}[S \mbox{ occurs}, D_N^c],
\end{align}
where $c(S)$ denotes the number of separating components of $S$.
By item (1) of Theorem \ref{thm:local_analyticity} and Claim $A$, there exists a complex neighbourhood of $p$ in which all the terms in the sum above can be extended as analytic functions. 
Moreover, item (2) of Theorem \ref{thm:local_analyticity} combined with Claims A and B yields the existence of $\eps > 0$ such that, for all $|z| < \eps$ and $n > 0$,
\begin{align*}
	\sum\limits_{S \in MS_n^N}|\phi_{p_0+z,q}[S \mbox{ occurs}, D_N^c]| \leq \exp(-(t/2) n),
\end{align*}
where $N$ and $t$ are positive constants given by Claim B. So,
\begin{align}
	\sum\limits_{n \geq 1} \sup\limits_{|z| < \eps} \sum\limits_{S \in MS_n^N}|\phi_{p_0+z,q}[S \mbox{ occurs}, D_N^c]| < \infty.
\end{align}

Hence, by the Weierstrass M-test, the series given by \eqref{eq:extension_theta_sep_comp} is uniformly convergent on $\D_\eps(0)$. So, it defines an analytic extension of $\phi_{p,q}[\mbox{Some separating component occurs}, D_N^c]$ in a complex neighbourhood of $p_0$. By Theorem \ref{thm:local_analyticity}, the function $\phi_{p,q}[D_N]$ also admits a analytic extension to a neighbourhood of $p_0$ since $D_N$ is a local event. It follows from Claim~C that $p \mapsto \theta(p)$ admits an analytic extension around $p_0$. 

The result for the magnetisation $m^*$ in the Potts model is a direct consequence of the result for $\theta$ and the Edwards-Sokal coupling relation~\eqref{equ:edwards_sokal} relating $m^*$ with $\theta$.

\end{proof}

\subsection{Analyticity of multi-point correlation functions}

We hope to have convinced the reader that once Theorem~\ref{thm:local_analyticity} holds, all the arguments of~\cite{PG_analyticity} can be replicated \emph{mutatis mutandis}, by relying on uniform local uniqueness in place of independence of Bernoulli percolation. We thus simply refer to the latter work for a proof of Theorem~\ref{thm:multi_point_correlations}.

\begin{proof}[Sketch of proof of Theorem~\ref{thm:multi_point_correlations}] 
We exactly follow the argument of~\cite[Theorem 8.10]{PG_analyticity}, where the same statement is established for Bernoulli percolation.  
All the combinatorial statements can be imported, and the independence of Bernoulli percolation is replaced by the uniformity of the probability of the local uniqueness in the boundary conditions given by (2b). 
Finally, we use Theorem~\ref{thm:local_analyticity} in the end of the proof to provide uniformly bounded extensions of $\tau$ and $\tau^\mathsf{f}$. 
\end{proof}

\vspace{0,5cm}
\subsection*{Acknowledgements}
The authors wish to thank Sébastien Ott for suggesting a different resummation than our original one, and thus improving the presentation of Lemma~\ref{lem:plus_minus_one_expansion}.
We also thank Christoforos Panagiotis for many insightful discussions related to this work. 
The research of LD is supported by the French National Research Agency (ANR), project number \texttt{ANR-23-CPJ1-0150-01}. The research of LG is supported by the Swiss National Science Foundation.

\newpage
\bibliographystyle{amsalpha}
\bibliography{biblioanalyticity.org.tug}
\newpage
\setcounter{section}{1}
\setcounter{counterEnvMain}{1}
\setcounter{counterEnvDefault}{1}

\renewcommand{\thesection}{\Alph{section}}

\renewcommand{\thecounterEnvMain}{\thesection.\arabic{counterEnvMain}}
\renewcommand{\thecounterEnvDefault}{\thesection.\arabic{counterEnvDefault}}

\setcounter{counterEnvMain}{0}
\setcounter{counterEnvDefault}{0}
\section*{Appendix: cluster expansion tools}

We recall some classical results about cluster expansion techniques, together with technical lemmas that are used throughout the paper. 
The statements are not as general as in~\cite{velenik_book}, but are rather tailored for our setting. 
The framework is that of~\emph{non-intersecting} polymer models on $\T_N^L$, with activities depending on a parameter.

Recall the following definitions, introduced in Definition \ref{def:polymer}: 

A \emph{polymer} $\gamma \subset \T_N^L$ is a non-empty connected subset of $\T_N^L$. The set of polymers is denoted by $\mathsf{P}_N^L$ and the set of \emph{non-intersecting} families of polymers is denoted by $\mathsf{NI}$.
Also recall that $c(d) > 0$ is a constant such that the cardinality of the set of polymers $\gamma$ satisfying $0 \in \gamma$ and $|\gamma| = n$ is upper bounded by $\exp(c(d)n)$. 

An \emph{activity function} is a function $\sfw : \mathsf{P}_N^L \rightarrow \C$, and we will consider families of activity functions $\sfw_z$ indexed by $z \in \bbC$. 
Recall that the polymer partition function is defined as 
\[
\Xi(\sfw_z) = \sum_{\{\gamma\} \in \mathsf{NI}} \prod_{\gamma \in \{\gamma\}} \sfw_z(\gamma). 
\]

Finally, for any $n \geq 0$, the $n$-th \emph{Ursell function} is classically defined as follows. For an ordered family of polymers $(\gamma_1, \cdots, \gamma_n)$, define
\[
\varphi_n(\gamma_1, \dots, \gamma_n) \coloneqq  \frac{1}{n!}\sum_{G \subset K_n} \prod_{\{ij\} \in G} \{-\ind{\gamma_i \cap \gamma_j \neq \emptyset}\}.
\]
The sum is taken over all the \emph{connected} subgraphs of the complete graph $K_n$ with $n$ vertices. 

In this setting, the main theorem is the following\footnote{We insist that this result is much less general than the full power of the cluster expansion given by~\cite{velenik_book}. In particular, the hardcore interaction of the polymers and the convergence criterion that we propose are far from being optimal.}

\begin{theorem}\label{thm:cluster_expansion}
Assume that for $\epsilon > 0$ small enough, there exists a constant $C_\eps > 0$, going to 0 when $\eps$ tends to 0, such that for any $|z| < \eps$, for every polymer $\gamma \in \sfP^L_N$, 
\[
|\sfw_z(\gamma)| \leq C_\eps \exp( - |\gamma|(1 + c(d))).
\]
Then, there exists $\eps > 0$ such that:
\begin{itemize}
\item For all $|z| < \eps$, the following equality is true, and the sum appearing in the right-hand side converges absolutely:
\[
\Xi(\sfw_z) = \exp\big( \sum_{n\geq 1}\sum_{\gamma_1, \cdots, \gamma_n \in \mathsf{P}_N^L} \varphi_n(\gamma_1, \dots, \gamma_n) \prod_{k = 1}^{n} \sfw_z(\gamma_k) \big). 
\]
We then say that $\Xi(\sfw_z)$ admits a \textit{convergent cluster expansion}.
\item Moreover, there exists $c_\eps > 0$, going to 0 when $\eps$ goes to 0 such that for any $|z| < \eps$, and all $N \in \bbN$,
\[
\sum_{n\geq 1}\sum_{\gamma_1, \cdots, \gamma_n \in \mathsf{P}_N^L} \big|\varphi_n(\gamma_1, \dots, \gamma_n)\big| \prod_{k = 1}^{n} \big|\sfw_z(\gamma_k)\big| \leq c_{\eps} |\bbT_N^L| . 
\]
\end{itemize}
\end{theorem}

The following result is particularly useful in this work. 
To the best knowledge of the authors, the proof proposed below is new. 

\begin{corollary}\label{thm:cluster_expansion_inter_A}
Assume that for $\epsilon > 0$ small enough, there exists a constant $C_\eps > 0$, going to 0 when $\eps$ tends to 0, such that for any $|z| < \eps$, for every polymer $\gamma \in \sfP^L_N$, 
\[
|\sfw_z(\gamma)| \leq C_\eps \exp( - |\gamma|(2 + c(d))).
\]
Then, there exists $c_\eps > 0$, going to 0 when $\eps$ goes to 0 such that for any $|z| < \eps$, for any set $A \subset \T_N^L$, 
\[
\sum_{n\geq 1}\sum_{\gamma_1, \cdots, \gamma_n \in \mathsf{P}_N^L \atop A \cap (\gamma_1 \cup \cdots \cup \gamma_n) \neq \emptyset} \big\vert\varphi_n(\gamma_1, \dots, \gamma_n)\big\vert \prod_{k = 1}^{n} \big\vert \sfw_z(\gamma_k)\big\vert \leq c_\eps |A|.
\]
\end{corollary}

\begin{proof}[Proof of Theorem \ref{thm:cluster_expansion}]
The proof of the first statement consists in checking the criterion provided by \cite[Theorem 5.4]{velenik_book}.
Indeed, first consider $\eps$ small enough so that the inequality $C_\eps (1 - \exp( -1/2 ))^{-1} \leq 1/2$ holds. 
Further set $a(\gamma) \coloneqq  \tfrac{1}{2}|\gamma|$.
Then, for any $|z| < \eps$ and any $\gamma' \in \sfP^L_N$,
\begin{align}
\sum_{\gamma \in \Gamma^N_L }|\sfw_z(\gamma)|\ind{\gamma \cap \gamma' \neq \emptyset}\e^{a(\gamma)} 
&\leq \sum_{x \in \gamma'}\sum_{\gamma \ni x }|\sfw_z(\gamma)|\e^{a(\gamma)} \\
&\leq |\gamma'|\sum_{k=1}^{|\bbT_N^L|} C_{\eps} \e^{kc(d)} \e^{-k(1 + c(d))}\e^{k/2} \\
&\leq C_{\eps}|\gamma'| (1 - \exp( -1/2 )) ^{-1} \leq a(\gamma'). \label{eq:convergence_criterion}  
\end{align}
 The convergence criterion of \cite[Theorem 5.4]{velenik_book} is thus satisfied and the first item of the theorem is true. 
 Now, for the second statement, \cite[Theorem 5.4]{velenik_book} also implies that for any $\gamma_1 \in \mathsf{P}_N^L$
\[
\vert \sfw_z(\gamma_1) \vert \Big(1 + \sum\limits_{n \geq 2} n \sum\limits_{\gamma_2,\dots,\gamma_n \in \mathsf{P}_N^L} \big\vert\varphi_n(\gamma_1, \dots, \gamma_n)\big\vert \prod_{k = 2}^{n} \big\vert\sfw_z(\gamma_k) \big\vert \Big) \leq \vert \sfw_z(\gamma_1) \vert e^{a(\gamma_1)}.
\]
Summing over all $\gamma_1 \in \mathsf{P}_N^L$, we get
\[
\sum_{n\geq 1}\sum_{\gamma_1, \cdots, \gamma_n \in \mathsf{P}_N^L} n \big|\varphi_n(\gamma_1, \dots, \gamma_n)\big| \prod_{k = 1}^{n} \big|\sfw_z(\gamma_k)\big| \leq \sum\limits_{\gamma_1 \in \mathsf{P}_N^L} |\sfw_z(\gamma_1)| e^{a(\gamma_1)} . 
\]
Taking $\gamma' = \bbT_N^L$ in \eqref{eq:convergence_criterion}, then yields
\[
\sum_{n\geq 1}\sum_{\gamma_1, \cdots, \gamma_n \in \mathsf{P}_N^L} n \big|\varphi_n(\gamma_1, \dots, \gamma_n)\big| \prod_{k = 1}^{n} \big|\sfw_z(\gamma_k)\big| \leq C_{\eps}|\bbT_N^L|(1 - \exp( -1/2 ))^{-1},
\]
which concludes the proof.

\end{proof}

\begin{proof}[Proof of Corollary \ref{thm:cluster_expansion_inter_A}]
We use an ``averaging method'' to obtain the required bound. 
 First assume that $\eps$ is small enough so that the cluster expansion is convergent. 
 Fix $A \subset \T_N^L$.
 In that case, the quantity
 \[
 \sum_{n\geq 1}\sum_{\gamma_1, \cdots, \gamma_n \atop A \cap (\gamma_1 \cup \cdots \cup \gamma_n) \neq \emptyset} \varphi_n(\gamma_1, \dots, \gamma_n) \prod_{k = 1}^{n} \sfw_z(\gamma)
 \]
 is well-defined by the absolute convergence of the cluster expansion (by Theorem~\ref{thm:cluster_expansion}).
 The next step is to observe that both the Ursell functions $U$ and the activities $\sfw_z$ are invariant under translations of the polymer family. Thus, 
 \begin{multline*}
 \sum_{n\geq 1}\sum_{\gamma_1, \cdots, \gamma_n \atop A \cap (\gamma_1 \cup \cdots \cup \gamma_n) \neq \emptyset} \big\vert \varphi_n(\gamma_1, \dots, \gamma_n) \big\vert \prod_{k = 1}^{n} \big\vert \sfw_z(\gamma_k) \big\vert \\
= \frac{1}{|\T_N^L|} \sum_{\tau \in \T_N^L}   \sum_{n\geq 1}\sum_{\gamma_1, \cdots, \gamma_n \atop (\tau +A) \cap (\gamma_1 \cup \cdots \cup \gamma_n) \neq \emptyset} \big\vert \varphi_n(\gamma_1, \dots, \gamma_n) \big\vert \prod_{k = 1}^{n} \big\vert \sfw_z(\gamma_k) \big\vert.
 \end{multline*}
By rearranging the terms
\begin{multline*}
\sum_{n\geq 1}\sum_{\gamma_1, \cdots, \gamma_n \atop A \cap (\gamma_1 \cup \cdots \cup \gamma_n) \neq \emptyset} \big\vert \varphi_n(\gamma_1, \dots, \gamma_n) \big\vert \prod_{k = 1}^{n} \big\vert \sfw_z(\gamma_k) \big\vert\\
 = \frac{1}{|\T_N^L|} \sum_{n\geq 1}\sum_{\gamma_1, \cdots, \gamma_n} \big|\varphi_n(\gamma_1, \dots, \gamma_n)\big|  \sum_{\tau \in \T_N^L \atop (\tau +A) \cap (\gamma_1 \cup \cdots \cup \gamma_n) \neq \emptyset} \prod_{k = 1}^{n} \big|\sfw_z(\gamma_k)\big| .
\end{multline*}

 Observe that in the third sum of the right-hand side, at most $|A|\cdot |\gamma_1 \cup \cdots \cup \gamma_n|$ choices of $\tau$ may produce a non-zero term. 
 We use the bound $ |\gamma_1 \cup \cdots \cup \gamma_n | \leq \prod_{k=1}^n (1 + |\gamma_k|)$ to obtain:
 \begin{multline}
\sum_{n\geq 1}\sum_{\gamma_1, \cdots, \gamma_n \atop A \cap (\gamma_1 \cup \cdots \cup \gamma_n) \neq \emptyset} \big\vert \varphi_n(\gamma_1, \dots, \gamma_n) \big\vert \prod_{k = 1}^{n} \big\vert \sfw_z(\gamma_k) \big\vert \\
 \leq \frac{|A|}{|\bbT_N^L|} \sum_{n\geq 1}\sum_{\gamma_1, \cdots, \gamma_n} |\varphi_n(\gamma_1, \dots, \gamma_n)| \prod_{k = 1}^{n}(1+|\gamma_k|)\cdot |\sfw_z(\gamma_k)|. \label{ineq:averaged_bound}
 \end{multline}
Note that $\sum_{n\geq 1}\sum_{\gamma_1, \cdots, \gamma_n} \varphi_n(\gamma_1, \dots, \gamma_n) \prod_{k \geq 1}(1+|\gamma_k|)\cdot \sfw_z(\gamma_k)$ is the cluster expansion of a polymer partition function with modified activities $\tilde\sfw_z(\gamma) \coloneqq  (1+|\gamma|)\cdot \sfw_z(\gamma)$. 
Observe that  these modified activities satisfy the requirement of Theorem \ref{thm:cluster_expansion}. So, by the second point of this theorem, we obtain that, for all $N \in \bbN$,
\[
\sum_{n\geq 1}\sum_{\gamma_1, \cdots, \gamma_n \in \mathsf{P}_N^L} |\varphi_n(\gamma_1, \dots, \gamma_n)| \prod_{k = 1}^{n}(1+|\gamma_k|)\cdot |\sfw_z(\gamma_k)| \leq c_{\eps} |\bbT_N^L|,
\]
with $c_{\eps}$ that goes to $0$ as $\eps$ goes to $0$. Substituting this in \eqref{ineq:averaged_bound} gives
\[
\sum_{n\geq 1}\sum_{\gamma_1, \cdots, \gamma_n \atop A \cap (\gamma_1 \cup \cdots \cup \gamma_n) \neq \emptyset} \big\vert \varphi_n(\gamma_1, \dots, \gamma_n) \big\vert \prod_{k = 1}^{n} \big\vert \sfw_z(\gamma_k) \big\vert \leq c_{\eps} |A|,
\]
which concludes the proof.

\end{proof}

\end{document}